\documentclass[12pt]{amsart}

\usepackage{amsmath}
\usepackage{amssymb}
\usepackage{amstext}
\usepackage{amscd}
\usepackage[frenchb,english]{babel}
\usepackage{enumitem}
\usepackage{greekctr}
\usepackage[matrix,arrow,ps]{xy}
\usepackage[bookmarks=true,colorlinks=true,citecolor=blue,urlcolor=blue,linkcolor=magenta]{hyperref}

\newtheorem{theorem}{Theorem}[section]
\newtheorem{proposition}[theorem]{Proposition}
\newtheorem{lemma}[theorem]{Lemma}
\newtheorem{corollary}[theorem]{Corollary}
\newtheorem{conjecture}[theorem]{Conjecture}
\newtheorem{definition}[theorem]{Definition}

\newcommand{\Hom}{\mbox{Hom}}

\newcommand{\End}{{\rm End}}

\newcommand{\GL}{{\rm GL}}

\newcommand{\Aut}{{\rm Aut}}
\newcommand{\Supp}{{\rm Supp}}

\newcommand{\Gal}{{\rm Gal}}
\newcommand{\tors}{{\text{tors}}}

\newcommand{\F}{{\mathbb F}}

\newcommand{\R}{{\mathbb R}}

\newcommand{\Z}{{\mathbb Z}}

\newcommand{\Q}{{\mathbb Q}}
\newcommand{\C}{{\mathbb C}}

\newcommand{\A}{{\mathbb A}}

\renewcommand{\Gal}{{\operatorname{Gal}}}

\newcommand{\eps}{{\varepsilon}}

\newcommand{\ol}{\overline}
\newcommand{\wh}{\widehat}

\newcommand{\wt}[1]{{\widetilde{#1}}}

\DeclareMathOperator{\coker}{\operatorname{coker}}

\newcommand{\tens}{\otimes}
\usepackage{mathtools}

\DeclarePairedDelimiterX{\Nm}[1]{\left\lVert}{\right\rVert}{#1}
\DeclarePairedDelimiterX{\norm}[1]{\lVert}{\rVert}{#1}
\newcommand{\sous}{\backslash}

\newcommand{\abs}[1]{\left|#1\right|}

\title[Hybrid conjecture in a mixed Shimura variety]{Hybrid conjecture in a mixed Shimura variety}
\author{Rodolphe Richard}
\email{Rodolphe.Richard@normalesup.org}
\address{The University of Manchester,
Alan Turing Building,
Oxford Rd,
Manchester
M13 9PL,
UK}
\author{Andrei Yafaev}
\email{a.yafaev@ucl.ac.uk}
\address{UCL, Department of Mathematics, Gower Street, WC1E 6BT, London, UK}
\date{\today}
\begin{document}
\setcounter{tocdepth}{1}
\setcounter{secnumdepth}{4}
\maketitle

\begin{abstract}The authors previously formulated the hybrid conjecture, unifying  André-Pink-Zannier and André-Oort conjectures, and proved it in Shimura varieties of abelian type. We study its analogue for mixed Shimura varieties, and consider the prime example, the universal abelian scheme~$\mathcal{A}_g\to \mathbb{A}_g$.

In a radical departure from the Pila-Zannier strategy, typically applied to such questions, we employ instead a combination of equidistribution and o-minimality

Our main result strictly includes the following:  the Hybrid Conjecture, in particular the André-Pink-Zannier and André-Oort conjectures, for~$\mathbb{A}_g$; the mixed André-Oort conjecture for~$\mathcal{A}_g$; and Manin-Mumford conjecture for arbitrary abelian varieties. It also yields an analogue of the ``Manin-Mumford in arithmetic pencil", a result of \cite{BRU}, for abelian schemes over a variety.

The mixed hybrid conjecture in~$\mathcal{A}_g$ also encompasses the Mordell-Lang conjecture. We actually reduce the mixed hybrid conjecture for~$\mathbb{A}_g$ to its "mordellic" part. 

We also prove, Galois-theoretic results: uniform variants on the Ribet's Kummer theory of Abelian varieties, and Serre's theorem on Lang's conjecture.
\end{abstract}

\tableofcontents

\section{Introduction}

Let~$g\in\Z_{\geq0}$, let~$\mathbf{A}_g$ be the moduli space of principally polarised abelian varieties of dimension~$g$ over~$\C$.
We have~$\mathbf{A}_g(\C)\simeq Sp(2g,\Z)\backslash \mathfrak{H}_g$ where~$\mathfrak{H}_g$ is
 the Siegel upper half-space. Let~$\Gamma:=\ker (Sp(2g,\Z)\to Sp(2g,\Z/(3)))$ or any other torsion free congruence subgroup. Then $\mathbf{A}_\Gamma:=\Gamma\backslash \mathfrak{H}_g$ is a \emph{fine}
moduli space of abelian varieties equipped with some appropriate level structure of type~$\Gamma$. There exists a corresponding universal abelian 
scheme~$\mathcal{A}_\Gamma\to  \mathbf{A}_\Gamma$. 

Then~$\mathbf{A}_g$ is the archetype of a Shimura variety and~$\mathcal{A}_\Gamma$ is the archetype of a mixed Shimura variety.

In~\cite{RY:LMS1}, the authors introduced the hybrid conjecture, which encompass both the André-Oort conjecture and the André-Pink-Zannier conjecture, using the notion of \emph{hybrid orbit} in a Shimura variety, and, in~\cite{RY:LMS2} they proved it for Shimura varieties of abelian type.
Here, we define the analogue in the mixed Shimura variety~$\mathcal{A}_\Gamma$, which we call the \emph{mixed hybrid orbit}. 

This notion unifies the hybrid conjecture for~$\mathbf{A}_g$ -- which includes the André-Oort conjecture (AO) and the generalised André-Pink-Zannier conjecture (APZ) --  with the Mordell-Lang conjecture (ML) -- which includes the Mordell conjecture (MC) and the Manin-Mumford conjecture (MM) -- for an arbitrary abelian variety. This also generalises the mixed André-Oort conjecture (MxAO) for~$\mathcal{A}_\Gamma$.

Our main results are Theorems~\ref{thm:Main} and~\ref{thm:MMonHyb}. Theorem~\ref{thm:MMonHyb} implies all the former results:~(HC), (AO), (APZ), (MM), (MxAO), except the mordellic part of the Mordell-Lang conjecture:  (MC) and thus (ML). Theorem~\ref{thm:MMonHyb} is a particular case of Theorem~\ref{thm:Main}, which reduces the ``mixed hybrid conjecture" for~$\mathcal{A}_\Gamma$ to its "mordellic part".

This is a combination of two results, interesting on their own. 

A first result, Theorem~\ref{thm:weak} uses a combination of equidistribution and o-minimality, introduced by the first named author and E. Ullmo in their ergodic-theoretic proof of the geometric part of the André-Oort conjecture. Theorem~\ref{thm:weak} implies the analogue, for abelian scheme over a variety of characteristic zero, of a result of~\cite{BRU}, ``Manin-Mumford conjecture in artihmetic pencils'', which concerns abelian schemes over rings of algebraic integers.

Secondly are results on Galois representations attached to abelian varieties. Theorem~\ref{thm:Serre hybrid} is a reinforcement of Serre's Theorem~\ref{thm:Serre eAK} on a conjecture of Lang on homotheties of the Tate modules. Our version is uniform on a hybrid orbit of an abelian variety, which generalises the isogeny class by allowing CM factors and passing to quotient and to powers (cf. Lemma~\ref{lem:truc}). Our next result is Theorem~\ref{thm:uniform Kummer1}, which is uniform version, over a hybrid orbit of an abelian variety, of Ribet's Kummer theory~Th.~\ref{thm:Kummer} for Abelian varieties. We actually prove stronger results. Theorem~\ref{thm:Serre unif principal} goes beyond Lang's conjecture, considering any element in the centre of the Mumford-Tate group, not only homotheties. Theorem~\ref{thm:uniform Kummer1} proves uniformity for arbitrary abelian extension, not only those coming by CM theory. We also give a new formulation, Theorem~\ref{thm:uniform Faltings}, of Tate uniform integral conjecture, derived from Faltings' theorem on Tate conjectures, of independent interest. 
These Galois theoretic results are not only over number fields, but also for finitely generated extension of~$\Q$. 

Because our approach is based on equidistribution, it may lead to refinements, though we did not explore that question. For instance, for cases in which one can prove topologic refinements of the Hybrid conjecture in the Simura variety~$\mathbf{A}_g$ using equidistribution (e.g. \cite{Duke,Ilya,RY:SHecke}), does our approach deduce to corresponding refinements in the mixed Shimura variety~$\mathcal{A}_g$?

\subsection{Main statements}
When~$A$ is a principally polarised abelian variety over~$\C$, and~$Q\in A(\C)$, and~$\lambda$ is a $\Gamma$-level structure\footnote{There is~$N\in\Z_{\geq1}$ such that~$\Gamma\geq \Gamma(N):=\ker (Sp(2g,\Z)\to Sp(2g,\Z/(N))$, and a~$\Gamma$-level structure on~$A$ amounts to an right coset~$\lambda\in \left(Isom(\Z/(N))^{2g},A[N])\right)/\Gamma$.} on~$A$, we denote by~$A\in \mathbf{A}_g(\C)$, by~$(A,\lambda)\in \mathbf{A}_\Gamma(\C)$ and~$(A,Q,\lambda)\in \mathcal{A}_\Gamma(\C)$ the corresponding moduli point.


In the following, we consider that abelian varieties of dimension zero have CM.
\begin{definition}[Hybrid, mixed hybrid and mordellic orbits]
\label{def:mixed hybrid orbits}
Let~$A$ be an abelian variety over~$\C$ and let~$P\in A(\C)$. 
We define
\begin{subequations}
\begin{multline}
\Sigma_\Gamma(A):=\{(B,\lambda)\in \mathbf{A}_\Gamma|\exists n\in \Z_{\geq1},
\exists \phi:A^n\to B,\\ B/\phi(A^n)\text{ has CM}\}.
\end{multline}
\begin{multline}\label{eq:def Hgamma}
H_\Gamma(A,P):=\{(B,Q,\lambda)\in \mathcal{A}_\Gamma|\exists d,n\in \Z_{\geq1}, \exists \phi:A^n\to B, 
\\B/\phi(A^n)\text{ has CM and }d\cdot Q=\phi(P,\ldots,P)\}.
\end{multline}
\begin{multline}
M_\Gamma(A,P):=\{(B,Q,\lambda)\in \mathcal{A}_\Gamma|\exists n\in \Z_{\geq1},\exists \phi:A^n\to B,\\ B/\phi(A^n)\text{ has CM and }Q=\phi(P,\ldots,P)\}.
\end{multline}
\end{subequations}
We say that~$H_\Gamma(A,P)\subseteq \mathcal{A}_g$ is the \emph{mixed hybrid orbit of~$(A,P)$
in~$\mathcal{A}_\Gamma$}
and that~$M_\Gamma(A,P)\subseteq H_\Gamma(A,P)\subseteq \mathcal{A}_\Gamma$ is the \emph{mordellic orbit of~$(A,P)$ in~$\mathcal{A}_\Gamma$}.

\end{definition}


The relation with the Hybrid orbit defined in~\cite{RY:LMS1,RY:LMS2} is the following.
\begin{lemma}\label{lem:truc}
Let~$(GSp(2g),\pm\mathfrak{H}_g)$ denote Siegel's Shimura datum. Let~$M_A$ be the Mumford-Tate group of a principally polarised~$A\in \mathbf{A}_g(\C)$ and let~$x_A:\C^\times\to M_A(\R)$
be its Hodge cocharater, and let~$X_A=M_A(\R)\cdot x_A$.
Let~$\Sigma_{\pm\mathfrak{H}_g}(M_A,X_A,x_A)$ be as in~\cite[Def. 6]{RY:LMS1}. Then
\[
\Sigma_g(A)\subseteq GSp(2g,\Z)\backslash\Sigma_{\pm\mathfrak{H}_g}(M_A,X_A,x_A).
\]
\end{lemma}
\begin{proof}Let~$A$ be an abelian variety over~$\C$. Let~$A\to \bigoplus_{i=1}^d A_i^{n_i}$ be an isogeny where the~$A_i$ are simple and pairwise non-isogeneous and~$n_i\geq 1$. Let~$x_i:=x_{A_i}^{ad}:\C^\times \to MT(A_i)\to M_i:=MT(A_i)^{ad}$ be the morphism induced by the Hodge cocharacter for~$A_i$. We recall that the image of
\[
(x_1,\ldots,x_d):\C^\times\to M_1\times\ldots\times M_d
\]
is~$\Q$-Zariski dense. Let~$\xi=\bigoplus x_i^{\oplus n_i}:\C^\times \to M_1^{n_1}\times\ldots\times M_d^{n_d}$.
Let~$M'_i\leq M_i^{n_i}$ be the image of~$g\mapsto (g,\ldots,g):M_i\to M_i^{n_i}$. Then the image of~$\xi$ is $\Q$-Zariski dense in
\[
M':=M_1'\times\ldots\times M_d'\leq M_1^{n_1}\times\ldots\times M_d^{n_d}.
\]
We claim that there is a natural isomorphism~$\iota:M^{ad}_A\simeq M'$ such that~$\xi=\iota\circ x^{ad}_A$ where~$x^{ad}_A:\C^\times\to M_A\to M_A^{ad}$ is induced by the Hodge cocharacter of~$A$. Using the isomorphisms~$M_i\to M'_i$ one can identify~$x_A^{ad}$ with~$(x_1,\ldots,x_d):\C^\times\to M_1\times\ldots \times M_d$.

 Consider an abelian variety~$B$ over~$\C$ and let similarly~$B\to \bigoplus_{j=1}^e B_j^{m_j}$ be an isogeny, such that  the~$B_j$ simple and pairwise non-isogeneous.

From the above discussion one sees that there exists
\[
\phi:M^{ad}_A\to M^{ad}_B
\]
 such that~$x_B^{ad}=\phi\circ x_A^{ad}$ if and only: for every~$1\leq j\leq e$ such that $M^{ad}_{B_j}$ is not trivial, there is~$1\leq i\leq d$ such that~$A_i$ is isogeneous to~$B_j$. This is readily equivalent to the property
\[
\exists n\in \Z_{\geq1},
\exists \phi:A^n\to B,\\ B/\phi(A^n)\text{ has CM}.
\]
The Lemma follows.
\end{proof}
The proof of Lemma~\ref{lem:truc} gives the following interpretation of hybrid orbits for abelian varieties.
The abelian variety~$B$ is in te hybrid orbit of~$A$ if and only if every simple quotient of~$B$ is CM or is also a quotient of~$A$.

The main result of~\cite{RY:LMS2} implies the following. We refer to~\cite{RY:LMS2} for the definition of a weakly special subvariety in a pure Shimura variety.
\begin{theorem}[Hybrid conjecture for~$\A_\Gamma$ \cite{RY:LMS2}]\label{thm:IHES}
Let~$\Sigma\subseteq \Sigma_g(A)$ be a subset. Then the irreducible components of~$\ol{\Sigma}^{Zar}$ are weakly special subvarieties.
\end{theorem}

Theorem~\ref{thm:IHES} establishes part~\eqref{conj:0} of the following Conjecture.
\begin{conjecture}[Mixed hybrid conjecture for universal abelian schemes]
\label{conj}
Let~$A$ be an abelian variety over~$\C$, and~$P\in A(\C)$. Let~$H_{\Gamma}(A,P)$ be as in~\eqref{eq:def Hgamma} and let~$\Sigma\subseteq H_{\Gamma}(A,P)$ be a subset. Let~$V\subseteq \ol{\Sigma}^{Zar}$ be an irreducible component and let~$S\subseteq \A_\Gamma$ be the image of~$V$.
\begin{enumerate}\setcounter{enumi}{0}
\item \label{conj:0} Then~$S\subseteq \A_\Gamma$ is a weakly special subvariety.
\item \label{conj:1}
Assume that for some~$s\in S$ we have~$V\cap (A_s)_{tors}\neq \emptyset$. Then there exists~$d\in \Z_{\geq1}$ and an abelian subscheme~$\mathcal{B}\leq \mathcal{A}_{\Gamma}|_S$ such that~$d\cdot V=\mathcal{B}$. Equivalently,~$V$ is an irreducible component of~$\mathcal{B}+\mathcal{A}_{\Gamma}|_S[d]$.
\item \label{conj:2}
In general, let~$\eta\in S$ be the generic point of~$S$ and~$V_\eta\subseteq \mathcal{A}_\eta$ be the fibres of~$V\subseteq \mathcal{A}_\Gamma$ over~$\eta$.  
Then there exists~$d\in \Z_{\geq1}$ and~$\phi:A\to \mathcal{A}_\eta$ and an abelian subvariety~$B\leq A_\eta$ such that~$d\cdot V_\eta=\phi(P)+B\subseteq \mathcal{A}_\eta$.
\end{enumerate}
\end{conjecture}
\subsubsection{}\label{sec:mordell lang case}

If~$S$ is a singleton~$\{s\}$, and~$A=\mathcal{A}^n_s$ and~$P=(P_1,\ldots,P_n)$, then Conjecture~\ref{conj} implies the \emph{Mordell-Lang conjecture} for the finitely generated subgroup
\[\{Q\in \mathcal{A}_s|\exists m\in\Z_{\geq1},\phi_1,\ldots,\phi_n\in\End(\mathcal{A}_s), mQ=\phi_1(P_1)+\ldots+\phi_n(P_n)\}\leq \mathcal{A}_s.\]
\subsubsection{}

Assume~$A=0$. Therefore~$P=0$ and~$\Sigma$ is the set of~$(B,Q,\lambda)$ in $\mathcal{A}_\Gamma$ such that~$B$ is CM and~$Q$ is a torsion point. That case of Conjecture~\ref{conj} is the \emph{mixed André-Oort conjecture} for~$\mathcal{A}_\Gamma$. It implies Manin-Mumford conjecture only for CM abelian varieties.
\subsubsection{}\label{sec:conj for P0}	
Assume~$P=0$ for an arbitrary~$A$. Then Conjecture~\ref{conj} implies (\cite{Pink})
\begin{itemize}
\item the \emph{Hybrid conjecture} in~$\mathbf{A}_g$ (\cite{RY:LMS2}), and in particular
\begin{itemize}
\item  the \emph{André-Oort Conjecture} for~$\mathbf{A}_g$ (\cite{PT,Tsimerman,AGHM,YZ}), and
\item  the \emph{generalised André-Pink-Zannier Conjecture} for~$\mathbf{A}_g$ (\cite{RY:IHES});
\end{itemize}
\item the \emph{Manin-Mumford conjecture} for arbitrary abelian varieties;
\item the \emph{mixed André-Oort Conjecture} in~$\mathcal{A}_g$ (cf.\,\cite{Gao}).
\end{itemize}

\subsubsection{Results}
Our main result Theorem~\ref{thm:Main} is the following.

\begin{theorem}\label{thm:Main}
Let~$A$,~$P$,~$\Sigma$, $V$ and $S$ be as in Conjecture~\ref{conj} and $M_\Gamma(A,P)$ be as in Definition~\ref{def:mixed hybrid orbits}.
%

Then there exists~$d\in \Z_{\geq1}$ and an abelian subscheme~$\mathcal{B}\leq \mathcal{A}_{\Gamma}|_{S}$ such that~$d\cdot V=W+\mathcal{B}$ with
\[
W=\ol{W\cap M_\Gamma(A,P)}^{Zar}.
\]
\end{theorem}
When~$P=0$, the set~$M_g(A,P)$ is contained in the neutral section~$\mathbf{A}_g\to \mathcal{A}_g$ and for every~$n$, the point~$Q_n$ has finite order in~$B_n$. 
Theorem~\ref{thm:Main} gives the following, which encompasses all the Conjectures mentionned in \S~\ref{sec:conj for P0}.
\begin{corollary}\label{thm:MMonHyb} Conjecture~\ref{conj} holds if~$P=0$. We are moreover in the case~\eqref{conj:1} of Conjecture~\ref{conj}.
\end{corollary}
	
Note that Theorem~\ref{thm:Main} reduces Conjecture~\ref{conj} to its ``Mordellic part''.
\begin{corollary}It is enough to prove Conjecture~\ref{conj} for~$\Sigma$ such that~$\Sigma\subseteq M_\Gamma(A,P)$.
\end{corollary}
\subsection{The Mordellic case}
Two of the main strategies used to tackle problems of Zilber-Pink type are: the Pila-Zannier strategy; and Equidistribution.
Both approaches rely on "large Galois orbits".
Moreover, the Pila-Zannier strategy needs quantitative lower bounds, in terms of the height of the points one considers, for the cardinality of the Galois orbits.

Let us consider the case~\S\ref{sec:mordell lang case}, when~$P$ is not a torsion point. Let~$K/\Q$ be a finitely generated extension such that~$A$ as a model over~$K$ such that~$P\in A(K)$, and let~$g=\dim(A)$. Then~$M_\Gamma(A,P)$ contains the infinite subset~$\{(A,n\cdot P,\lambda)|n\in\Z\}\subseteq \mathcal{A}_\Gamma(K)$, and~$H_\Gamma(A,P)$ contains 
\[
\{(A,Q,\lambda)|\exists p\in \Z,q \in\Z\smallsetminus\{0\}, q\cdot Q=p\cdot P\}
\]

When~$gcd(p,q)=1$, the height of~$(A,Q,\lambda)$ is roughly~$\max\{p;q\}$. Then
\begin{itemize}
\item Pila-Zannier strategy applies to sequences~$(A,Q_n,\lambda)$ such that~$q_n\to \infty$ and~$p_n\leq f(q_n)$, with~$f(x)=a+x^b$ for some~$a,b\in\R$, or some other explicit~$f:\R\to \R$ in finer variant of the strategy.
\item Equidistribution can only apply to sequences~$(A,Q_n,\lambda)$ such that~$q_n\to \infty$.
\end{itemize}
One can of course find, for every~$f:\R\to\R$, a sequence~$(p_n,q_n)$ such that~$p_n>f(q_n)$ for~$n\gg0$.

In this setting, the methods of this article apply to all sequences for which~$q_n\to +\infty$. This is what allow us to reduce the general case to the case where~$q_n$ is bounded, or ultimately~$q_n=1$. Theorem~\ref{thm:Main} is something the Pila-Zannier strategy may fail to achieve.

In the case~$\#S=1$, Theorem~\ref{thm:Main} reduces the Mordell-Lang conjecture for abelian varieties to its Mordellic part~\cite{Fal91} by Faltings. This reduction was originally achieved in~\cite{McQ} by McQuillan (which also covers semi-abelian varieties), using a method introduced by Hindry for the Manin-Mumford conjecture.

In that context Theorem~\ref{thm:Main} recovers results of~\cite{McQ} using an entirely new method. Our method  builds on the Equidistribution proof~\cite{RWeyl} of Manin-Mumford by Weyl criterion, and combines equidistribution with o-minimality following a trick introduced by~\cite{RU}. 

\subsection{Strategy}

Our proof of Theorem~\ref{thm:Main} combines two ingredients.
\subsubsection{Equidistibution} The first is based on Theorem~\ref{thm:equi0}, which is adapted from a trick introduced by the first author in Ullmo in~\cite{RU}. 

The equidistribution method, cf. \cite{U}, traditionally proceeds as follows. We consider a sequence of subsets~$E_n\subseteq S$ in an ambiant variety~$S$. The goal is to determine~$V:=\ol{\bigcup E_n}^{Zar}\subseteq S$. One constructs, for each~$n$, a measure~$\mu_n\in Prob(S)$ with support~$E_n$. One studies the limits~$\mu_\infty$. A difficulty is that one needs to establish:
\begin{enumerate}
\item \label{enum:equi1}
that the limit measure~$\mu_\infty$ is non zero.
\item \label{enum:equi2}
that~$\Supp(\mu_\infty)$ contains~$\bigcup E_n$.
\end{enumerate}
One can then infer that~$V=\ol{Supp(\mu_\infty)}^{Zar}$.

Theorem~\ref{thm:equi0} allows us to obtain information on~$V$ when~\eqref{enum:equi2} fails, or even when~$\mu_\infty=0$. This argument can be applied when the situation, for some o-minimal structure, is related to definable families of definable subsets.

A typical example is studied in this article, and this aticle can serve as an introduction to this new paradigm of the equidistribution method.

In the universal abelian scheme~$S=\mathcal{A}_\Gamma$, consider a sequence~$s_n\in \A_\Gamma$, and subsets~$E_n\subseteq A_{s_n}$ in the corresponding fibers. Note that~$E_n\not\subseteq \Supp(\mu_\infty)$, except when~$s_n$ is stationnary. Note that~$\A_\Gamma$ is not compact, and when~$s_n$ diverges to infinity, we necessarily have~$\mu_\infty=0$.

Using these ideas, we obtain Theorem~\ref{thm:weak}.
Theorem~\ref{thm:weak} is an analogue of~\cite[1.9]{BRU}. The latter studies "Manin-Mumford conjecture" in arithmetic pencils, whereas we study Mordell-Lang conjecture for an abelian scheme over a complex algebraic variety.

\subsubsection{Galois representations}

In \S\ref{sec:Galois} we consider two Galois representations: the Tate module attached to an abelian variety, and the affine Tate module (cf. \S\ref{sec:Kummer}) attached to~$(A,P)$ for some~$P\in A$.

The first main result of~\S\ref{sec:Galois}  Theorem~\ref{thm:Serre hybrid} is a stronger variant of Serre's Theorem~\ref{thm:Serre eAK} that is uniform for abelian varieties ranging through a hybrid orbit. Theorem~\ref{thm:Serre eAK} concerns homoteties~$\lambda\in GL(1,\wh{\Z})$ of a Tate module which are in the image of the Galois representation. 
Actually, we obtain a Theorem~\ref{thm:Serre unif principal}
which concerns all the elements~$\lambda\in Z(M)(\wh{\Z})$ in the centre~$Z(M)$ of the Mumford-Tate group of~$A$.

The second main result Theorem~\ref{thm:uniform Kummer1} is a stronger variant of Ribet's Theorem~\ref{thm:Kummer} on Kummer theory for abelian varieties, that is uniform for abelian varieties ranging through a hybrid orbit (cf. Theorem~\ref{thm:Gal En}). We adapt a presentation~\cite[App. 2]{Hindry} of a proof of Ribet's theorem. The critical difference is our use of Sah's theorem Sah's~\cite[App. 2, Lem.~D]{Hindry} in the proof of Proposition~\ref{prop:412}. We also needed to fully develop Hindry's arguments and make them quantitative. This allows us to obtain a precise quantitative form~\eqref{eq:Kummer explicite} of Ribet's theorem. For instance, we obtain (cf. \S\ref{sec:CD}) a stronger form of a Theorem of Rémond that is a "key tool" (\cite[p. 16]{CD}) in a recent article~\cite{CD} of Checcoli and Dill. 

We also give a reformulation Theorem~\ref{thm:uniform Faltings} of Faltings' theorem on Tate's conjecture.

\subsubsection{Combining the ingredients} In \S\ref{sec:proof}, we prove theorem~\ref{thm:Main}. We consider~$(A_{s_n},Q_n)\in H_\Gamma(A,P)$. Using \S\ref{sec:Galois}, we construct subsets~$E_n\subseteq A_{s_n}\cap Gal(\ol{K}/K)\cdot (A_{s_n},Q_n)$. This allows us to apply Theorem~\ref{thm:weak}. We finish the proof by comparing the conclusion of Theorem~\ref{thm:weak} with the property~$(A_{s_n},d\cdot Q_n)\in M_\Gamma(A,P)$.

\subsection{Acknowledgements}The first named author was supported by the University of Manchester and Grant EP/Y020758/1.

\section{Equidistribution in o-minimal families}
In this section, we use the term \emph{definable} with respect to a chosen~$o$-minimal structure (see~\cite{VdD}).
\begin{theorem}\label{thm:equi0}
Let~$X\times Y$ be a definable set, let~$y_n\in Y$ be a sequence, let~$\nu_n$ be a sequence of probability measures in~$X$, and assume that~$\nu_n$ has a limit~$\nu_\infty$ and that
\begin{equation}\label{eq:equi1}
\forall n\gg0, Supp(\nu_n)\subseteq Supp(\nu_\infty),
\end{equation}
and that~$Supp(\nu_\infty)$ is definable and that for any definable subset~$A\subseteq Supp(\nu_\infty)$ such that~$\dim(A)<\dim(Supp(\nu_\infty))$ one has~$\nu_\infty(A)=0$. 

\begin{itemize}
\item
Let~$V\subseteq X\times Y$ be a definable subset such that
\begin{equation}\label{eq:equi2}
\forall n\gg0, Supp(\nu_n)\times\{y_n\}\subseteq V.
\end{equation}
Then, for~$n\gg0$, 
\begin{equation}\label{eq:equi3}
\dim\Bigl(Supp(\nu_\infty)\times\{y_n\}\,\cap\,V\Bigr)=\dim(Supp(\nu_\infty)).
\end{equation}In particular,~$V$ contains a non-empty open subset of~$(Supp(\nu_\infty)+x_n)\times \{y_n\}$.
\item
Assume that~$X$ has a definable Lie group structure~$(x_1,x_2)\mapsto x_1\cdot x_2$, and let~$x_n\to x_\infty$ be a convergent sequence in~$X$.

Let~$V\subseteq X\times Y$ be a definable subset such that
\begin{equation}\label{eq:equi4}
\forall n\gg0, (Supp(\nu_n)\cdot x_n)\times\{y_n\}\subseteq V.
\end{equation}
Then, for~$n\gg0$, 
\begin{equation}\label{eq:equi6}
\dim\Bigl(\left(Supp(\nu_n)\cdot x_n\right)\times\{y_n\}\,\cap\,V\Bigr)=\dim(Supp(\nu_\infty)).
\end{equation}
In particular,~$V$ contains a non-empty open subset of~$(Supp(\nu_\infty)+x_n)\times \{y_n\}$.
\end{itemize}
\end{theorem}
\begin{proof}Without loss of generality we may assume that~$X$ is a bounded subset of~$\R^n$.
We apply~\cite[Th.~4.3, Cor.~4.4]{RU} with~$K=\ol{X}$; with~$b_i=0\in B:=\{0\}$ and~$A(b)=Supp(\nu_\infty)$; with~$b'_i=y_n$ in~$Y=B'$; with~$A'(b'_i)\times \{b'_i\}=V_i:=V\cap \left(X\times \{b'_i\}\right)$. The assumption~$\nu_\infty(A)=0$ of Theorem~\ref{thm:equi} implies the assumption~\cite[(4.3)]{RU}, and~\eqref{eq:equi1} implies the assumption~\cite[(4.3)]{RU}. Our assumption~\eqref{eq:equi2} gives~$\nu_n(V_n)=1$, and this implies the assumption~\cite[(4.4)]{RU}. Let~$d:=\dim(Supp(\nu_\infty))$. Then the assumption~\cite[(4.4)]{RU} is satisfied.

We deduce
\[
\forall n\gg0, d=\dim(\Supp(\nu_\infty)\cap V_n).
\]
This proves~\eqref{eq:equi3}.

Let us now prove~\eqref{eq:equi6}.

Let~$Y'=X\times Y$ and~$y'_n=(x_n,y_n)$. Let~$V'=\{(v',x,y)|(v'\cdot x,y)\in V\}$, which is definable. Let~$V'_{(x,y)}$ be such that~$V'\cap X\times \{(x,y)\}=V'_{(x,y)}\times \{(x,y)\}$ and let~$V_y$ be such that~$V_y\times\{x\}=V\cap \{y\}$. We have~$V'_{(x,y)}\cdot x=V_y$. By~\eqref{eq:equi4}, we have
\[
Supp(\nu_n)\cdot x_n\subseteq V_{y_n}= V'_{(x_n,y_n)}\cdot x_n.
\]
Therefore~$Supp(\nu_n)\subseteq V'_{(x_n,y_n)}$. 
We apply~\eqref{eq:equi3} to~$V'\subseteq X\times Y'$ and~$y'_n=(x_n,y_n)$. We have
\begin{equation}
\forall n\gg0, \dim\Bigl(Supp(\nu_\infty)\times\{(x_n,y_n)\}\,\cap\,V'\Bigr)=\dim(Supp(\nu_\infty)).
\end{equation}
In other terms,~$\dim(Supp(\nu_\infty)\cap V'_{(x_n,y_n)})=\dim(Supp(\nu_\infty))$, or
\[
\dim(Supp(\nu_\infty)\cap V_{y_n}\cdot x_n^{-1})=\dim(Supp(\nu_\infty)).
\]
We observe~$\dim(Supp(\nu_\infty)\cap V_{y_n}\cdot x_n^{-1})=\dim(Supp(\nu_\infty)\cdot x_n\cap V_{y_n})$. This implies~\eqref{eq:equi6} and concludes the proof of Theorem~\ref{thm:equi}.
\end{proof}

\subsection{Some equidistribution statements}\label{sec:31}

Consider~$K=\R^N/\Z^N$ and~$e\in\Z_{\geq1}$.
For a real Lie subgroup~$G\leq K$, let~$\nu_{G}$ be the Haar probability measure associated with~$G$, viewed as a probability measure on~$K$ concentrated on~$G$. For~$k\in K$, we denote by~$\nu_{G+k}$ the image of~$\nu_G$ by the translation~$k'\mapsto k+k':K\to K$. Let~$L_e:=\{\lambda^e|\lambda\in\wh{\Z}^{\times}\}$ act on~$K_{tors}=\Q^N/\Z^N$. For~$G\leq K$ a real Lie subgroup, and~$t\in K_{tors}$ and~$k\in K$ we define
\[
E(G,t,k)=\bigcup_{t'\in L_e\cdot t} G+t'+k
\]
and
\[
\mu_E:=\frac{1}{\#L_e\cdot t}\sum_{t'\in L_e\cdot t} \nu_{G+t'+k}
\]

The following implies that the set of measures of the form~$\mu_{E(G,t,k)}$ is a compact subset in the space of probabilities on~$K$.
\begin{theorem}\label{thm:equi}
Let~$G_i$ be a sequence of real Lie subgroups of~$K$, let~$t_i\in K_{tors}=\Q^N/\Z^N$ be
a sequence of torsion points, and let~$k_i$ be a sequence in~$K$.

Assume that the sequence~$\mu_{E(G_i,t_i,k_i)}$ has a limit~$\mu_\infty$ in the space of probabilities on~$K$.

Then there exists~$G_\infty\leq K$, and~$t_\infty\in K_{tors}$ and~$k_\infty\in K$ such that
\[
\mu_\infty=\mu_{E(G_\infty,t_\infty,k_\infty)}.
\]
Moreover, if the sequence~$k_i$ has a limit~$k_\infty$ in~$K$, then
\[
\mu_{E(G_i,t_i,0)}\to \mu_{E(G_\infty,t_\infty,0)}
\]
and
\[
\forall i\gg0, E(G_i,t_i,0)\subseteq E(G_\infty,t_\infty,0).
\]
\end{theorem}

Let~$\ast$ denote the convolution operation on probability measures on~$K$.
We observe that
\[
\mu_{E(G,t,k)}=
\mu_{E(G,0,0)}
\ast
\mu_{E(0,t,0)}
\ast
\mu_{E(0,0,k)}.
\]
Assume that~$\mu_{E(G_i,t_i,k_i)}$ has a limit~$\mu_\infty$. Since~$K$ is compact,
after possibly extracting a subsequence, there exist limits
\[
\alpha=\lim \mu_{E(G_i,0,0)}\qquad
\beta=\lim \mu_{E(0,t_i,0)}\qquad
\gamma=\lim \mu_{E(0,0,k_i)}.
\]
By continuity, we have~$\mu_\infty=\alpha\ast\beta\ast\gamma$. We note that~$
\alpha=\lim \mu_{E(0,0,k_i)}$ is the Dirac mass~$\delta_{k_i}$ and converges if and only if~$k_i$
has a limit~$k_\infty$ in~$K$. We have then
\[
\gamma=\delta_{k_{\infty}}=\mu_{E(0,0,k_\infty)}.
\]

By~\cite{RWeyl}, we have
\[
\beta=\mu_{E(B,t_\infty,0)}
\]
for some~$B\leq K$ and~$t_\infty\in K_{tors}$ such that
\[
\forall i\gg0, L_e\cdot t_i\subseteq B+L_e\cdot t_\infty.
\]

Moreover, we claim that, for some Lie subgroup~$G_{\infty}\leq K$, we have
\[
\alpha=\mu_{E(G_\infty,0,0)}\qquad\forall i\gg0, G_i\leq G_\infty.
\]
\begin{proof}Let~$X=Hom(K,U(1))\simeq \Z^n$. For~$\chi\in X$, let~$m_\chi(\nu_{G_i})=\int \chi(x) d\nu_{G_i}(x)$.
Let~$\Lambda_i:=\{\chi\in X|\chi(G_i)=\{1\}\}$. Then~$\chi\mapsto m_\chi(\nu_{G_i}):X\to \C$ is the indicator function~$1_{\Lambda_i}$ of~$\Lambda_i$. By Weyl criterion the function~$f:\chi\mapsto m_\chi(\alpha):X\to \C$ is the pointwise limit of~$1_{\Lambda_i}$. For~$\chi,\chi'\in X$ and~$i\in\Z_{\geq1}$ 
\[
m_\chi(\nu_{G_i})\in\{0;1\}.
\]
It follows that~$f(\chi)\in\{0;1\}$ and~$f(\chi)=1\Leftrightarrow\exists j_{\chi}\in\Z_{\geq1} \forall i\geq j_{\chi}, \chi\in\Lambda_i$.
Assume~$f(\chi)=f(\chi')=1$. For~$i\geq \max\{j_\chi;j_\chi'\}$, we have~$\chi',\chi\in\Lambda_i$.
Since each~$\Lambda_i$ is a group, we have
\[
\forall i\geq \max\{j_\chi;j_\chi'\}, \chi-\chi'\in \Lambda_i.
\]
We deduce~$f(\chi-\chi')=1$. Moreover, for~$\chi=1:X\to U(1)$, we have~$f(1)=1$. We deduce that~$\Lambda=\{\chi|f(\chi)=1\}$ is a subgroup of~$X$. Let~$G_\infty:=\{k\in K|\forall \chi\in\Lambda,\chi(k)=1\}$. Then
\[
m_\chi(\alpha)=f(\chi)=1_{\Lambda}(\chi)=m_\chi(\nu_{G_\infty}).
\]
By Weyl criterion we have
\[
\lim \mu_{E(G_i,0,0)}
=
\alpha
=\nu_{G_{\infty}}.
\]
Let~$\chi_1,\ldots,\chi_r$ be generators of~$\Lambda$. Then, for~$i\geq \max\{j_{\chi_1};\ldots;j_{\chi_r}\}$, we have
\[
\chi_1,\ldots,\chi_r \in \Lambda_i.
\]
Since~$\Lambda_i$ is a subgroup, we have~$\Lambda\leq \Lambda_i$. This implies~$G_i\leq G_{\infty}$.
\end{proof}

\section{On weak Mordell-Lang in geometric families}

\begin{theorem}\label{thm:weak}
Let~$A\to S$ be an abelian scheme over an algebraic variety over~$\C$ and~$e\in\Z_{\geq1}$. We assume that, 
\begin{equation}\label{eq:hypmonodromy}
\text{for some~$s\in S$, the monodromy action of~$\pi_1(S(\C))$ on~$\End(A_s)$ is trivial.}
\end{equation}
Let~$e\in\Z_{\geq1}$ and let~$L_e:=\{\lambda^e|\lambda\in\wh{\Z}^\times\}$ act on every commutative torsion group.

Let~$s_n$ be a generic sequence in~$S$ and  let~$A_n$ denote the fibre above~$s_n$. For every~$n$, let~$a_n\in (A_n)_{tors}$  be torsion point, let $G_n\leq A_n$  an algebraic subgroup and let~$k_n\in A_n$ be point of~$A_n$. Let
\[E_n:=G_n+L_e\cdot a_n+k_n
=\{g+\lambda^e\cdot a_n+k_n|g\in G_n, \lambda\in \wh{\Z}\}\subseteq A_n(\C)
\]
and let~$V$ be the Zariski closure of~$\bigcup E_n$ in~$A$.

Then for every irreducible component~$C$ of~$V$, there exists an abelian subscheme~$B\leq A$ such that~$C=C+B$
and~$\#E_n/B_{s_n}=O(1)$.
\end{theorem}

\subsection{Definable trivilialisation}
We recall the following.
\begin{theorem}[{\cite[Theorem~1.2]{PS}}]\label{thm:PS} Let~$\alpha:A\to S$ be an abelian scheme over an irreducible complex algebraic variety.

There exists a~$\R_{an,\exp}$-definable map~$\tau:A(\C)\to U(1)^{2g}$ such that for every~$s\in S$, the map~$\tau_s:A_s(\C)\to A(\C)\to U(1)^{2g}$ is a Lie group isomorphism. We say that~$\tau$ is a definable trivialisation of~$A\to S$
\end{theorem}
\begin{proof}Let~$S'\to S$ be a morphism of algebraic varieties, and let~$\alpha':A':=A\times_S S'\to S'$ be the abelian $S'$-scheme obtained by base change. 

We observe that~$\tau:A(\C)\to U(1)^{2g}$ is a definable trivialisation of~$A\to S$, then~$A'\to A\to U(1)^{2g}$ is a trivialisation of~$A'\to S'$. 

If~$S'\to S$ is surjective, then, by definable choice, there exists a definable section~$\sigma:S(\C)\to S'(\C)$ of~$S'(\C)\to S(\C)$. We note that~$A_s=A'_{\sigma(s)}$ by definition. If~$\tau':A'(\C)\to U(1)^{2g}$ is a trivialisation of~$A'\to S'$, then the map~$\tau:A(\C)\to U(1)^{2g}$ defined by~$\forall s\in S, \forall a\in A_s(\C), \tau(a)=\tau'(a)$ is a definable trivialisation of~$A\to S$.

Let~$S'\to S$ be an étale cover such that~$A'\to S'$ admits a polarisation and a full level structure of level~$N=3$.
Then there is a moduli map~$m:S'\to \mathbf{A}_{g,\psi}(N)$ such that~$A'\to S'$, the pullback of~$A$ by~$S'\to S$, is also the pullback of the universal abelian scheme~$\mathcal{A}\to \mathbf{A}_{g,\psi}(N)$. By the observation above, it is enough to treat the case where~$A\to S$ is~$\mathcal{A}\to \mathbf{A}_{g,\psi}(N)$.

The Theorem is then~\cite[Theorem~1.2]{PS}.
\end{proof}

We will prove the following.
\begin{theorem}\label{thm:locus}
Let~$A\to S$ and~$\tau:A\to U(1)^{2g}$ be as in Theorem~\ref{thm:PS}.
Let~$s_n$ be a generic sequence in~$S$ and~$B_n\leq A_n:=A_{s_n}$ be abelian subvarieties such that~$H:=\tau_{s_n}(B_n)\leq  U(1)^{2g}$ does not depend on~$n$. Assume~\eqref{eq:hypmonodromy}.

Then, possibly after extracting an infinite subsequence, the Zariski closure~$\ol{\bigcup B_n}^{Zar}$ is an abelian subscheme~$B\hookrightarrow A\to S$.
\end{theorem}
\subsubsection{}
We recall some facts while introducing some notations. Let~$\psi:\Z^{2g}\times\Z^{2g}\to \Z$ be a bilinear form such that its~$\Q$-linear extension~$\psi_\Q:\Q^{2g}\times\Q^{2g}\to \Q$ is a nondegenerate symplectic form. We denote~$\Gamma(N):=GSp(\psi_\Q)\cap \ker(GL(2g,\Z)\to GL(2g,\Z/(N))$ the principal congruence subgroups. Let~${H_g}$ be the upper Siegel half-space and~$X:=\pm {H}_g$. An element~$\tau\in{H}_g$ is a symmetric~$g\times g$ matrix~$\tau=\Re(\tau)+i\cdot \Im(\tau):\C^g\to \C^g$ over~$\C$, such that~$\Im(\tau)$ is the matrix of a positive definite quadratic form on~$\R^g$. To~$\tau\in \mathcal{H}_g$ we attach the morphism of abelian groups:
\begin{equation}\label{eq:iota}
\iota_\tau:(I_g,\tau):\Z^{2g}\to \C^g,
\end{equation}
where~$I_g$ is the injection~$\Z^{g}\leq \C^g$.
Then~$\Lambda_\tau:=\iota_\tau(\Z^{2g})$ is a lattice of~$\C^g$ and~$A_\tau:=\C^g/\Lambda_\tau$ is an abelian variety, and the symplectic form~$\Lambda_\tau\times \Lambda_\tau\to 2\pi i\Z$ induced by~$2\pi i\psi$ is a polarisation on~$A_\tau$. The spaces
\[
\mathbf{A}_{g,\psi}(N):=\Gamma(N)\sous{H}_g
\]
are connected components of a Shimura varieties, for the Shimura datum~$(GSp(\psi_\Q),\pm H_g)$, which are moduli spaces of abelian varieties with polarisation of type~$\psi$ 
and level structure of level~$N$, and are smooth for~$N\gg 0$.

Consider the action of~$\C^\times$ on~$\C^g$ by homotheties and consider the structure of~$\C^g$ as a~$\R$-linear vector space. It defines a~$\R$-linear representation~$\C^\times \to GL_\R(\C^g)\simeq GL(2g,\R)$. Consider the isomorphism~$\iota_\tau\tens\R:\Z^{2g}\tens\R\to \C^g$. We deduce conjugated representation
\[
\C^\times \to Aut(\Z^{2g}\tens\R)\simeq GL(2g,\R) 
\]
and this representations factors via some~$h_\tau:\C^\times\to GSp(\psi_\Q)(\R)$.

Choose an~$\R$-Lie groups isomorphism~$K:=U(1)^{2g}\to \R^{2g}/\Z^{2g}$. We can then identify~$\Z^{2g}\simeq H^{1}(K;\Z)$, and isomorphisms~$K\simeq A_\tau$ of~$\R$-Lie groups corresponds to isomorphisms~$\Z^{2g}\simeq H_1(A_\tau;\Z)\simeq \Lambda_\tau$.  Let~$A_{H_g}:=\sqcup_\tau A_\tau\to H_g$ be the structural morphism of the holomorphic family of complex tori~$(A_\tau)_{\tau\in H_g}$. There is a natural quotient map~$\C^g\times H_g\to A_{H_g}$ and the topology on~$A_{H_g}$ is the quotient topology. Observe that~$H^1(A_{H_g}/H_g;\Z)\simeq \sqcup_\tau \Lambda_\tau\hookrightarrow \C^g\times H_g\to H_g$ is an étale cover and that~$\bigsqcup_{\tau\in H_g}:\iota_\tau:\Z^{2g}\times H_g\to H^1(A_{H_g}/H_g;\Z)$ is a tivialisation of this étale cover. Note also that~$\lambda:L:=\bigsqcup Isom(\Z^{2g},H_1(A_\tau;\Z))\to H_g$ is also an étale cover, and even a (right)~$GL(2g,\Z)$-torsor over~$H_g$. Let~$D\subseteq {H_g}$ be a connected subset. Every section~$\sigma:D\to L$ of~$\lambda$ over~$D$ defines a family of isomorphisms~$(j_{\sigma}(\tau):K\simeq A_\tau)_{\tau\in D}$. The isomorphisms~$\iota_\tau$ induce a distinguished section~$\sigma_0$.
\paragraph{}\label{para:facts}

We use the following facts:
\begin{itemize}
\item the map 
\[
\bigsqcup_{\tau \in D} j_\sigma(\tau):K\times D\to \bigsqcup_{\tau \in D} A_\tau
\]
is an homeomorphism if and only~$\sigma$ is continuous.
\item if~$D$ is connected and~$\sigma$ is continuous, the map~$\sigma$ is unique up to the action of~$GL(2g,\Z)$. More precisely there exists~$\gamma\in GL(2g,\Z)$ such that~$\forall \tau\in S, j_\sigma(\tau)=j_{\sigma_0}(\tau)\circ\gamma$.
\end{itemize}

We record a last few facts. A connected~$\R$-Lie subgroup has Lie subalgebra~$\mathfrak{b}=H_1(B;\Z)\tens\R$, and that~$B$ is an abelian subvariety if and only if~$\mathfrak{b}$ is stable under the action of~$\C^\times$. When this is the case, the symplectic form form induced by~$\psi_\Q$ on~$H_1(B;\Z)\tens\Q$ is  nondegenerate. Finally, the stabiliser in~$GSp(\psi)_\Q$ of a non degenerate subspace are all Levi subgroups, and are in particular reductive groups. 

\subsubsection{}

Before proving Theorem~\ref{thm:locus}, let us first establish the following.
\begin{lemma}\label{lem:locus}
Let~$H_g$ be Siegel upper half-space and~$\iota_\tau$ be as in~\eqref{eq:iota}. Let~$\Lambda\leq \Z^{2g}$ be a subgroup. Let~$X'$ be the set of~$\tau\in H_g$ such that~$\iota_\tau(\Lambda)\tens\R\leq \C^g$ is a~$\C$-linear subspace.

Assume that~$X'\neq\emptyset$.
Then there exist a Shimura subdatum~$(L,X)\leq (GSp(2g),\pm H_g)$ such that~$X'=X\cap H_g$.
\end{lemma}
\begin{proof}[Proof of Lemma~\ref{lem:locus}]
Consider~$\tau\in X'$. Let~$\psi:\Z^{2g}\times \Z^{2g}\to \Z$ denote the symplectic form associated with~$GSp(2g)$.
Then~$\C^g/\iota_\tau(\Lambda)\leq \C^g/(\Z^g+\tau\times\Z^g)$ is an abelian subvariety. We deduce that~$\Lambda$ is non-degenerate for~$\psi$. It follows that the stabiliser~$L\leq GSp(2g)_\Q$ of~$\Lambda\tens\Q$ is a Levi subgroup. In particular it is a reductive group. 

Let~$h_\tau:\C^\times\to GSp(2g)$ be the Hodge cocharacter corresponding to~$\tau$. Then~$h(\C^\times)\leq L(\R)^+$.
By~\cite[Prop. 3.2]{CU}, the pair~$(L,L(\R)\cdot \tau)$ is a Shimura subdatum of~$(GSp(2g),\pm H_g)$.

Consider~$\tau'\in X'$. Then~$L(\R)^+\cdot \tau$ and~$L(\R)^+\cdot \tau'$ are two~$L(\R)^+$ orbits which are symmetric subspaces of~$H_g$. This implies that there exists~$z\in Z_{GSp(2g,\R)^+}(L)$ such that
\[
L(\R)^+\cdot \tau'=z\cdot L(\R)^+\cdot \tau=L(\R)^+\cdot z\cdot \tau.
\]
But~$\tau$ is fixed by~$Z_{GSp(2g,\R)^+}(h_\tau(\C^\times))\geq Z_{GSp(2g,\R)^+}(L)$.
Therefore
\[
L(\R)^+\cdot \tau'=L(\R)^+\cdot \tau.
\]
Since~$\tau'\in X'$ is arbitrary, we have~$X'\subseteq L(\R)^+\cdot \tau$. On the other hand, one sees easily that~$X'\subseteq  L(\R)^+\cdot \tau$. The Lemma follows.
\end{proof}
We can now prove~Theorem~\ref{thm:locus}.
\begin{proof}[Proof of Theorem~\ref{thm:locus}]
We claim that there exists a definable partition~$S=S_1\sqcup\ldots\sqcup S_k$ such that
\begin{enumerate}
\item \label{enu1} the maps~$A_{S_i}:=\bigsqcup_{s\in S_i} A_s\to A\to U(1)^{2g}$ are continuous;
\item \label{enu2} there exists a definable~$\wt{S_i}\subseteq H_g$ such that~$\wt{S_i}\to H_g\to S_i$ is an homeomorphism;
\item \label{enu3} each~$S_i$ is connected and simply connected.
\end{enumerate}
\begin{proof}[Proof of the claim] By~\cite[Ch. 9]{VdD} we can find~$S_i$ such that~\eqref{enu1} holds true. After refining the partition, we may assume~\eqref{enu2}. By~\cite[Ch. 8]{VdD}, we may refine the partition in such a way that the~$S_i$ are simplexes, and in particular are connected and simply connected. This proves~\eqref{enu3}. 
\end{proof}
For~$s\in S_i$, corresponding to~$\wt{s}\in \wt{S_i}$, we have an isomorphism~$A_{\wt{s}}=\C^g/\Lambda_{\wt{s}}\to A_s$.
We choose an isomorphism~$K:=U(1)^{2g}\simeq \R^{2g}/\Z^{2g}$. We can identify the Lie algebra of~$U(1)^{2g}$ with~$\R^{2g}$ and~$H^1(K;\Z)$ with~$\Z^{2g}\leq\R^{2g}$. The map~$A_{\wt{s}}\to A_s\to A\to U(1)^{2g}$ induces a Lie algebra isomorphism
\[
\C^g\to \R^{2g}
\]
which induces an isomorphism~$j_{\wt{s}}:\Lambda_\wt{s}$ to~$\Z^{2g}$. We claim that there exists~$\phi\in Aut(U(1)^{2g})$ such that
\[
\forall \wt{s}\in \wt{S}, \iota_\wt{s}\circ\phi={j_{\wt{s}}}^{-1}
\]
\begin{proof}[Proof of the claim] See~\ref{para:facts}.
\end{proof}
Extracting a subsequence, we may assume that, for some~$i\in \{1;\ldots;k\}$, we have
\[
\forall n\in\Z_{\geq0}, s_n\in S_i.
\]
Let~$H$ as in the statement of Theorem~\ref{thm:locus}.
Let~$\Lambda=H^1(H;\Z)\leq \Z^{2g}=H^1(K;\Z)$ be the subgroup corresponding to~$H$.

Let~$X'\subseteq H_g$ be as in Lemma~\ref{lem:locus}. By assumption~$H=\tau_{s_n}(B_n)$ for an abelian 
subvariety~$B_n\leq A_{s_n}$. This implies~$\wt{s_n}\in X'$. Therefore,~$s_n$ is in the image~$Z$ of~$X'$ by~$H_g\to \mathbb{A}_g$. Lemma~\ref{lem:locus} implies that~$Z$ is a special subvariety of~$\mathbb{A}_g$. In particular~$Z\cap S$ Zariski closed in~$S$.

But~$s_n$ is a generic sequence in~$S$. Our extracted subsequence is thus generic in~$S$. This implies that~$\{s_n\}\subseteq  Z\cap S$ is Zariski dense in~$S$. As~$Z$ is closed, we have~$S\subseteq Z$.

We can deduce that~$\wt{S_i}$ is contained in~$X'$. By definition of~$X'$, the image of
\[
H\cap X'\to (\R^{2g}/\Z^{2g})\times X'\to \bigsqcup_{\tau \in X'} \C^{g}/\Lambda_\tau
\]
is an holomorphic family of abelian varieties. We deduce that its image in the universal abelian variety~$\mathcal{A}_g|_Z$ restricted to~$Z$ is an abelian sub-$Z$-scheme. Let~$B\to S$ be the restriction to~$S$ of this abelian subscheme. By construction~$B_n$ is equal to the fibre~$B_{s_n}$.

We deduce that~${\bigsqcup B_{s_n}}$ is the inverse image of the dense subset~$\{s_n\}$ by the map~$B\to S$.
Therefore
\[
\ol{\bigsqcup B_{s_n}}^{Zar}=\ol{B}^{Zar}=B.\qedhere
\]
\end{proof}
\subsection{Proof of Theorem~\ref{thm:weak}} Let~$a_n$,~$G_n$,~$k_n$,~$e$,~$E_n$ be as in the statement  of Theorem~\ref{thm:weak}, and let~$C$ be an irreductible component of the Zariski closure~$\ol{\bigcup E_n}^{Zar}$ in~$A$. Then, extracting a subsequence, we may assume that there exists~$c_n$ is a generic sequence in~$C$ such that~$c_n\in E_n\cap C$ for every~$n$. This implies that for every extraction~$\phi:\Z_{\geq0}\to \Z_{\geq0}$, the variety~$C$ is a component of~$\ol{\bigcup E_{\phi(n)}}^{Zar}$.

For every~$n$, let~$B_n\leq B'_n\leq A_n$ be the biggest abelian subvariety such that~$B'_n+E_n\subseteq \ol{\bigcup E_n}^{Zar}$. Substituting~$G_n$ with~$G_n+B'_n$, we may assume that~$B_n:=G^0_n$ is identical to~$B'_n$.

Let~$a'_n\in (A_n)_{tors}\cap (G_n+L_e\cdot a_n)$ be a torsion point of minimal order. We have
\[
G_n+L_e\cdot a_n=G_n+L_e\cdot a'_n
\]
and without loss of generality we may assume~$a_n=a'_n$.

Let~$\tau:A(\C) \to U(1)^{2g}$ be a definable trivialisation as in Theorem~\ref{thm:PS}. Let us define~$H_n:=\tau(G_n)$ and~$t_n:=\tau(a_n)$ and~$\kappa_n:=\tau(k_n)$. Then, with the notations of~\S\ref{sec:31}, we have
\[
\tau(E_n)=E(H_n,t_n,\kappa_n)
\]
and we can consider
\[
\mu_n:=\mu_{E(H_n,t_n,\kappa_n)}\text{ and }
\nu_n:=\mu_{E(H_n,t_n,0)}.
\]
Extracting a subsequence, we may assume that~$\kappa_n$ has a limit~$\kappa_{\infty}$ in~$U(1)^{2g}$ 
and that the sequence of probability measures~$\nu_n$ has a limit~$\nu_{\infty}$. By Theorem~\ref{thm:equi}
we have
\[
\nu_\infty=\mu_{E(H_\infty,t_\infty,0)}
\]
with~$t_\infty\in U(1)_{tors}^{2g}$ and~$H_\infty\leq U(1)^{2g}$ a real Lie subgroup such that
\begin{equation}\label{eq:truc}
\forall n\gg0, 
E(H_n,t_n,0)\subseteq E(H_\infty,t_\infty,0):= H_\infty +L_e\cdot t_{\infty}.
\end{equation}
Let~$V:=\ol{\bigcup E_n}^{Zar}$. Recall that~$C$ is an irreducible component of~$V$. Let~$a\mapsto s(a)$ denote the structural map~$A\to S$. Let~$V'\subseteq S(\C)\times U(1)^{2g}\times U(1)^{2g}$
be the image of~$V\times U(1)^{2g}$ by
\[
(v,g)\mapsto (s(v),g, \tau(v)-g).
\]
Note that~$V'\subseteq S(\C)\times U(1)^{2g}\times U(1)^{2g}$ is a definable subset. For every~$n\in\Z_{\geq0}$, we deduce from~$E_n\subseteq V(\C)$ that
we have
\begin{equation}\label{eq:ttruc}
\{s_n\}\times \{\kappa_n\}\times E(H_n,t_n,0)\subseteq V'.
\end{equation}
We apply Theorem~\ref{thm:equi0} with~$Y=S(\C)\times U(1)^{2g}$ and~$X=S(\C)\times U(1)^{2g}$ and~$y_n=(s_n,\kappa_n)$ and~$V\subseteq X\times Y$ corresponding to~$V'$ above. Then the assumptions~\eqref{eq:equi1} and~\eqref{eq:equi2} are satisfied by~\eqref{eq:truc} and~\eqref{eq:ttruc}. Note that~$\nu_\infty$ is a finite linear combination of translates of the Haar measure on~$H_\infty$, and thus the assumption~$\dim(A)<\dim(H_\infty)\Rightarrow \nu_\infty(A)=0$ is satisfied.
 
We deduce from~\eqref{eq:equi6} that, for~$n\gg0$, the set~$V'$ contains a non empty subset of~$\{s_n\}\times \{\kappa_n\}\times E(H_\infty,t_\infty,0)$ of dimension~$\dim(H_\infty)$. By construction
\[
V'\cap \left( \{s_n\}\times \{\kappa_n\}\times U(1)^{2g}\right)=\{s_n\}\times \{\kappa_n\}\times \tau((V\cap A_n)-k_n).
\]
Let~$G'_n\leq A_n(\C)$ be the real Lie subgroup such that~$\tau(G_n)=H_\infty$. 
Using the definable bijection
\[
(V\cap A_n)-k_n\simeq V'\cap \left(\{s_n\}\times \{\kappa_n\}\times U(1)^{2g}\right)
\] 
we deduce that~$V\cap A_n$ contains a subset of~$E'_n:=G'_n+L_e\cdot a_n+k_n$ of dimension~$\dim(H_\infty)=\dim(G_n)=\dim(E'_n)$. As in Theorem~\ref{thm:equi0}, we deduce that
\[
V\cap A_n
\]
contains, for the archimedean topology, an non empty open subset of~$E'_n$. As~$V\cap A_n$ is a Zariski closed subset,
we deduce that~$V\cap A_n$ contains a component of~$\ol{E'_n}^{Zar}=\ol{G'_n}^{Zar}+L_e\cdot a_n+k_n$.

Using similar arguments, one proves that~$V\cap A_n$ contains every component of~$\ol{G'_n}^{Zar}+L_e\cdot a_n+k_n$.

By our assumption~$B_n=B_n'$, we have 
\[
B_n=(\ol{G'_n}^{Zar})^0.
\]

We apply Theorem~\ref{thm:locus} to~$H=H^+_\infty$. Then, after possibly extracting a subsequence, we deduce that
\[
\ol{\bigcup B_n}
\]
is an abelian subscheme~$B$ of~$A\to S$.

Possibly extracting a subsequence, we may assume that~$c_n\in C$ does not belong to the intersection of~$C$ with the union of the other irreducible components of~$V$. Thus the inclusion~$c_n+B_n\subseteq V$ and the irreducibility of~$c_n+B_n$ imply that~$c_n+B_n\subseteq C$.

By Proposition~\ref{prop:truc} below, this proves~$C+B=C$.

\begin{proposition}\label{prop:truc} Let~$C$ be a subvariety of an abelian scheme~$A\to S$ and let~$B\leq A$ be an abelian subscheme. Let~$s_n$ and~$c_n$ be a generic sequence in~$S$ and in~$C$ with~$c_n$ in the fibre~$C_n:=C_{s_n}$. Assume 
\[
\forall n\geq 0, c_n+B_{s_n}\subseteq C.
\]
Then
\[
\forall s\in S, \forall c\in C_s, c+B_s\subseteq C_s.
\]
In other terms
\[
C=C+B.
\]
\end{proposition}
\begin{proof}
Let~$a:A\to S$ denote the structural map. It is enough to prove that the set
\[
\Sigma:=\{c\in C| c+B_{a(c)}\subseteq C\}
\]
describes a closed subvariety of~$C$: from~$\{c_n\}\subseteq \Sigma$, we conclude~$C=\ol{\{c_n\}}^{Zar}\subseteq \ol{\Sigma}^{Zar}\subseteq \ol{C}^{Zar}=C$.

For~$N\in \Z_{\geq0}$, let
\[
\Sigma_N:=\{c\in C| c+B_{a(c)}[N]\subseteq C\}.
\]
Recall that~$\bigcup_{N\in\Z_{\geq1}}B_{a(c)}[N]$ is dense in~$B_{a(c)}$. We deduce that for every~$c\in C$,
\[
\forall N\in\Z_{\geq0}, c+B_{a(c)}[N]\subseteq C\Rightarrow c+B_{a(c)}\subseteq C.
\]
It follows that~$\Sigma=\bigcap_{N\in\Z_{\geq0}} \Sigma_N$. It is thus enough to prove that~$\Sigma_N$ describes a closed subvariety of~$C$.

Recall that~$S':=B[N]$ is an étale cover of~$S$. Consider the base change~$C',B'\subseteq A'\to S'$ of~$C,B\subseteq A\to S$. Then~$B'[N]$ is a trivial cover of~$S'$: there exists sections~$\sigma_1,\ldots, \sigma_{N^{2\dim(B/S)}}:S'\to B'$ such that~$B'[N]=\bigcup_{1\leq i \leq N^{2\dim(B/S)}} \sigma_i(S')$.

Note that~$\Sigma_N$ is defined in terms of fibers. Let~$\Sigma'_N$ be the analogous set for~$C'$. Then~$\Sigma'_N$ is the pullback of~$\Sigma_N$, and it is enough to prove that~$\Sigma'_N$ describes a closed subvariety of~$C'$. 

Let~$i\in\{1;\ldots;N^{2\dim(B/S)}\}$ be arbitrary.
Note that the fibrewise sum~$C'+\sigma_i(S')$ is the image of the fibre product~$C'\times_{S'}\sigma_i(S')$ by the regular addition map~$A'\times_{S'}A'\to A'$. 
Recall that latter is a proper map, and that~$C'$ and~$S'$ are closed subvarieties of~$A'$. It follows that~$C'\times_{S'}\sigma_i(S')$ and~$C'+\sigma_i(S')$ are closed subvarieties of~$A'\times_{S'}A'$ and~$A'$.

One can check fibrewise that
\[
\Sigma'_N=\bigcap_{1\leq i \leq N^{2\dim(B/S)}}  C'+\sigma_i(S').
\]
This is an intersection of closed subvariety, and thus is itself a closed subvariety. The statement follows.\end{proof}

We have proved Proposition~\ref{prop:truc}. 

In order to prove Theorem~\ref{thm:weak}, it remains to prove
\[
\# E_n/B_{s_n}\leq O(1)\text{ as }n\to +\infty.
\]
Let~$c$ be the number of connected components of~$H_\infty$. We have
\begin{multline}
\# \left(E_n+B_{s_n}\right)/B_{s_n}
=\#\left(\tau(E_n)+\tau(B_{s_n})\right)/\tau(B_{s_n})\\
= \#\left(E(H_n,\tau_n,\kappa_n)+H^+_{\infty}\right)/H^+_{\infty}\\
\leq\#H_\infty/H_\infty^+\cdot \left(\#E(H_n,\tau_n,\kappa_n)+H_{\infty}\right)/H_{\infty}\\
=c\cdot \#E(H_\infty,t_\infty,\kappa_n)/H_{\infty}\\\leq c\cdot ord(t_\infty+H_{\infty}/H_{\infty}).
\end{multline}
This finishes the proof of Theorem~\ref{thm:weak}.

\section{Galois actions associated with a hybrid orbit of abelian varieties}\label{sec:Galois}
The main results of this section are Theorem~\ref{thm:Serre hybrid} and Theorem~\ref{thm:uniform Kummer1}.
Of independent interest is the reformulation Theorem~\ref{thm:uniform Faltings} of Falting's results
on Tate isogeny conjecture.

The strategy of this article to prove the main theorems is to find subset of Galois
orbits which are amenable to equidistribution results. We use two cases.
\begin{itemize}
\item The method of~\cite{RWeyl} applies to torsion points and uses Serre's Theorem~\ref{thm:Serre eAK}
on Lang's conjecture. 
\item The Galois action on division points of rational points which are not torsion points is the subject of Kummer theory for abelian varieties.
\end{itemize}
In both case we need a uniform results as our abelian variety varies in a hybrid orbit.
For Serre's theorem, this is Theorem~\ref{thm:Serre hybrid}, which relies on~\cite[App. B]{RY:IHES}'s approach to Serre's theorem. For Kummer theory, we follow~\cite{Hindry}, adding uniformity with respect to passing to a finite extension~$E/K$
and the maximal abelian extension~$E^{ab}/K$.

\subsection{Tate module}
Let~$A$ be an abelian variety over an extension~$K/\Q$ and let~$\ol{K}$ be an algebraic closure of~$K$. Denote by~$A_{tors}\leq A(\ol{K})$ the subgroup of torsion points.
The~$\wh{\Z}$-Tate module of~$A$ is then
\[
\wh{T}(A):=\Hom(\Q/\Z,A_{tors})\approx\wh{\Z}^{2\dim(A)}.
\]
The action of~$\Gal(\ol{K}/K)$ on~$A_{tors}$ induces a~$\wh{\Z}$-linear representation
\[
\rho_A:\Gal(\ol{K}/K)\to \Aut(\wh{T}(A))\approx GL(2\dim(A),\wh{\Z}).
\]
\subsubsection{Several theorems}We rely on the following.
\begin{theorem}[Mordell-Weil-Néron, {\cite[p. 52 Remark]{Serre:MWT}}, \cite{LN}]\label{thm:MWN}
Let~$A$  be an abelian variety over a finitely generated extension~$K$ over~$\Q$ or~$\F_p$.
Then~$A(K)$ is a finitely generated abelian group.
\end{theorem}

We observe that~$GL(1,\wh{\Z})$ acts naturally on any~$\wh{\Z}$-module, in particular torsion and any profinite abelian groups: $A_{tors}$, $A[N]$, $\wh{T}(A)$, \textit{etc}.

\begin{theorem}[Serre (see~{\cite[33, p.\,34, Th. 2 and 2']{Serre:O4}}, cf.~{\cite[Lemme 12]{Hindry}})]\label{thm:Serre eAK}
Let~$A\neq0$  be an abelian variety over a finitely generated extension~$K/\Q$. 
There exists~$e=e(A,K)\in\Z_{\geq1}$ such that
\begin{equation}
\label{eq:Serre e}
\forall \lambda\in GL(1,\wh{\Z}), \lambda^e\in \rho_A(Gal(\ol{K}/K)).
\end{equation}
\end{theorem}

The following is a reformulation of~\cite[Definition 2.1]{RY:IHES}.
\begin{theorem}[Uniform integral Tate conjecture, Faltings]\label{thm:uniform Faltings}
Let~$A$  be an abelian variety over a finitely generated extension~$K/\Q$ such that~$\End_A:=\End(A/K)=\End(A/\ol{K})$.

Let~$Z$ be the commutant of~$\End(A)$ in~$\End_{\wh{\Z}}(\wh{T}(A))$. Then for every~$D\geq 1$, there exists~$F(A,K,D)$ such that for every extension~$E/K$ of degree~$[E:K]\leq D$,
\begin{equation}\label{eq:TateUnifLinear}
\left[Z:\wh{\Z}\left[\rho_A(Gal(\ol{E}/E))\right]\right]\leq F(A,K,D)<+\infty.
\end{equation}
\end{theorem}
For~$K=E$ a number field,~\eqref{eq:TateUnifLinear} is~\cite[Théorème 2.7]{Deligne}.
This is also true when~$K$ is a finitely generated extension of~$\Q$.

For every prime~$\ell$, let~$F_\ell(A,E):=\left[Z\tens\Z_\ell:\Z_\ell\left[\rho_A(Gal(\ol{E}/E))\right]\right]$. We have
\[
\left[Z:\wh{\Z}\left[\rho_A(Gal(\ol{E}/E))\right]\right]=\prod_\ell F_\ell(A,E)=F(A,E,1).
\]
We will prove that there exists~$\ell(A,K,D)$ such that for every extension~$E/K$ of degree~$[E:K]\leq D$,
and every prime~$\ell\geq \ell(A,K,D)$, we have
\begin{equation}\label{eq:Tateclaim1}
F_\ell(A,E)=1
\end{equation}
We will also prove that, for every prime~$\ell$,
\begin{equation}\label{eq:Tateclaim2}
\sup_{E/K, [E:K]\leq D}F_\ell(A,E)<+\infty,
\end{equation}
where~$E$ ranges through the extensions of~$K$ of degree at most~$D$. Theorem~\ref{thm:uniform Faltings} follows from~\eqref{eq:Tateclaim1} together with \eqref{eq:Tateclaim2}.

Let us prove~\eqref{eq:Tateclaim2}. 
\begin{proof}{Proof of \eqref{eq:Tateclaim2}}Let~$\rho_{A,\ell}:Gal(\ol{E}/K)\to GL(\widehat{T}(A)\tens\Z_\ell)$ denote the natural representation on the~$\Z_\ell$-linear Tate module. We note that~$F_\ell(A,E)$ only depends on the subgroup~$\rho_{A,\ell}(Gal(\ol{E}/E))\leq \rho_{A,\ell}(Gal(\ol{E}/K))$. But the~$\ell$-adic group~$\rho_{A,\ell}(Gal(\ol{E}/K))$ has only finitely many closed subgroups of index at most~$D$.  Moreover
\begin{multline*}
\left[\rho_{A,\ell}(Gal(\ol{E}/E)): \rho_{A,\ell}(Gal(\ol{E}/K))\right]\leq \left[Gal(\ol{E}/E): Gal(\ol{E}/K)\right]
\\=[E:K]\leq D.
\end{multline*}
Therefore, in~\eqref{eq:Tateclaim2} the~$F_\ell(A,E)$ takes only finitely many values as~$E$ ranges through the extensions of~$K$ of degree at most~$D$. Note also that~$F_\ell(A,E)\leq F(A,E,1)<+\infty$. We deduce the finiteness in \eqref{eq:Tateclaim2}.
\end{proof}

\begin{proof}{Proof of \eqref{eq:Tateclaim1}}
For every prime~$\ell$, let~$\rho_{A[\ell]}$ be the natural~$\F_\ell$-linear representation~$Gal(\ol{K}/K)\to GL_{\wh{\Z}}(T(A))\to GL_{\F_\ell}(A[\ell])$.
If~$\ell \nmid F_\ell(A,K)$, then the image of~$Z$ in~$\End(A[\ell])$ is~$\F_\ell[\rho_{A[\ell]}(Gal(\ol{K}/K))]$.

Consider an embedding~$K\to \C$ and the corresponding~$\Z$-structure~$\Lambda:=H^1(A(\C);\Z)\tens\wh{\Z}\simeq T(A)$ with respect to some embedding~$\ol{K}\to \C$. Then~$R:=\End(A,\ol{K})$ embeds in~$\End_\Z(\Lambda)$ and its commutator~$Z'$ satisfies~$Z'\tens \wh{\Z}\simeq Z$. The algebras~$R\tens\Q$ and~$Z\tens\Q$ are semisimple. For~$\ell>\ell(A)$, the morphisms~$R\tens \F_\ell\to \End(A[\ell])$ and~$Z'\tens\F_\ell\simeq Z\tens\F_\ell\to\End(A[\ell])$ are injective; the images are semisimple~$\F_\ell$-algebras; and the images are the commutator of each other in~$\End(A[\ell])$.

Let~$G_\ell:=\rho_{A[\ell]}(Gal(\ol{K}/E))$.
For~$\ell\geq \ell(A,K,[E:K])$, \cite[Def. 2.1]{RY:IHES} implies that~$A[\ell]$ is a semisimple~$G_\ell$-module, and that~$\End(A)\tens\F_\ell$ is the algebra of~$G_\ell$-equivariant endomorphisms of~$\End(A)\tens\F_\ell$. This implies that~$\F_\ell[\rho_{A[\ell]}(Gal(\ol{K}/K))]=Z\tens\F_\ell=\F_\ell[\rho_{A[\ell]}(Gal(\ol{K}/E))]$. 

Recall that~$\rho_{A,\ell}(Gal(\ol{K}/K))\leq Z\tens\Z_\ell$. By Nakayama's lemma, we conclude~$\Z_\ell[\rho_A(Gal(\ol{K}/E))]=Z\tens\Z_\ell$ for~$\ell\geq \ell(A,K,[E:\Q])$.
\end{proof}
We have proved~\eqref{eq:Tateclaim1} and \eqref{eq:Tateclaim2}.
This concludes the proof of Theorem~\ref{thm:uniform Faltings}.

\subsection{Uniform variant of a result of Serre}
Here is a variant of Theorem~\ref{thm:Serre eAK}, which is uniform on a Hybrid orbit of abelian varieties.
\begin{theorem}[{\cite[loc. cit.]{Serre:O4}} and {\cite[App.~B]{RY:IHES}}]\label{thm:Serre hybrid}
Let~$A$ be an abelian variety over a finitely generated field extension~$K/\Q$ and define~$e(A,K)$ be the smallest~$e\in\Z_{\geq1}$ such that~\eqref{eq:Serre e}.
 Let~$B\leq A$ be an algebraic subgroup defined over~$\ol{K}$,  let~$K\leq K(B)\leq \ol{K}$ the smallest extension over which~$B$ is defined.
Then
\[\label{eq:isog unif}
e(B,K(B))\text{ and }e(A/B,K(B))\text{ divides }e(A,K).
\]

Let~$d\in\Z_{\geq1}$. There exists~$e(A,K,d)\in\Z_{\geq1}$ such that for every abelian variety~$C$ of dimension at most~$d$, with~$CM$ with reflex field~$E$, and defined over the Hilbert class field~$E(C)/E$, and for every algebraic subgroup~$F\leq A\times C$, 
\[
e((A\times C)/F,K\cdot E^\ast(C))\leq e(A,K,d),
\]
where~$Gal(\ol{\Q}/ E^\ast(C))$ is the inverse image of~$GL(1,\wh{\Z})$ by~$Gal(\ol{\Q}/ E(C))\to Aut((A\times C)_{tors})$.
\end{theorem}
\begin{proof}[Proof of Theorem~\ref{thm:Serre hybrid} assuming Theorem~\ref{thm:Serre unif principal}]
Let~$M$ be the Mumford-Tate group of~$A\times C$. We recall that there is an inclusion~$GL(1)\to Z(M)$ such that the~$\A_f$-linear representation of~$M(\A_f)$ on the~$\A_f$-Tate module of~$A\times C$ induces the action of~$GL(1,\A_f)$ by homotheties.

Let~$Z(M)(\wh{\Z})$ be as in Theorem~\ref{thm:Serre unif principal}. Let~$e:=e'(A,K,\dim(C))$ be as in Theorem~\ref{thm:Serre unif principal}. Then Theorem~\ref{thm:Serre unif principal} implies in particular that for every~$\lambda \in GL(1,\wh{\Z})$, the image of~$Gal(\ol{K}/K\cdot E(C))\to M(\A_f)$ contains~$\lambda^e$. Equivalently, the image of~$Gal(\ol{K}/K\cdot E(C))\to Aut((A\times C)_{tors})$ contains~$\lambda^e$.

By definition,~$\lambda^e$ is in the image of~$Gal(\ol{K}/K\cdot E^\ast(C))\leq Gal(\ol{K}/K\cdot E(C))\to Aut((A\times C)_{tors})$. 

Note that~$F$ is globally invariant under homotheties. Therefore~$(A\times C)/F$ and the quotient map~$A\times C\to (A\times C)/F$ are defined over~$K\cdot E^\ast(C)$. Let~$\sigma\in Gal(\ol{K}/K\cdot E^{\ast}(C))$ be such that~$\sigma$ act by the homothety~$\lambda^e$ on~$(A\times C)_{tors}$. Then~$\sigma$ act by the homothety~$\lambda^e$ on~$(A\times (C/F))_{tors}\simeq (A\times C)_{tors}/F_{tors}$.

This implies Theorem~\ref{thm:Serre hybrid}.
\end{proof}
%
%
%


Theorem~\ref{thm:Serre hybrid} is therefore a consequence of the following.
\begin{theorem}\label{thm:Serre unif principal}
 Assume that~$A$ is a product of simple factors, and that the CM factors of~$A$ are principal, and let~$C$ be a principal CM abelian variety. We denote by~$E$ the reflex field of~$C$ and by~$E(C)/E$ the Hilbert class field extension. We recall that~$C$ admits a model over~$E(C)$.

Let~$M$ be the Mumford-Tate group of~$A\times C$ and let~$Z(M)(\widehat{\Z})\leq Z(M)({\A_f})$ be the maximal compact subgroup. Let
\[
\rho_{A\times C}(Gal(\ol{K}/K\cdot E(C))\to M(\A_f)
\]
be the Galois representation induced by the Tate module of~$A\times C$, and let~$U\leq  M(\A_f)$ be its image.

Then there exists~$e':=e'(A,K,\dim(C))$ such that
\begin{equation}\label{eq:thm SUP}
\{z^{e'}| z\in Z(M)(\widehat{\Z})\}\subseteq U.
\end{equation}
\end{theorem}
The following was proved in~\cite[App. B]{RY:IHES}. Let~$M_A$ be the Mumford-Tate group of~$A$, and let~$U_{A/K}$ be the image of~$\rho_A:Gal(\ol{K}/K)\to M_A(\A_f)$. 
For every prime~$p$, let~$Y(p)$ be the image of~$U_{A/K}$ in~$M^{ad}(\Q_p)$, and let~$Y(p)^{ab}$ be its maximal abelian quotient. Then, by~\cite[App. B (110) and (111)]{RY:IHES}
\begin{itemize}
\item for every prime~$p$, we have
\begin{equation}\label{eq:110}
c(p,A,K):=\# Y(p)^{ab}<+\infty
\end{equation}
\item and there exists~$p_0=p_0(A,K)$ such that
\begin{equation}
\label{eq:111}
\forall p\geq p_0, \# Y(p)^{ab}=1.
\end{equation}
\end{itemize}
Moreover~\cite[App. B (113)]{RY:IHES} imply that there exists~$p_1(A,K)$ such that for every~$p\geq p_1(A,K)$, the group~$Y(p)$ is generated by topologically~$p$-nilpotent elements. In particular, for~$p\geq p_1(A,K)$ and every closed subgroup~$H\leq Y(p)$, we have
\begin{equation}\label{eq:113}
[Y(p):H]<p\Rightarrow
[Y(p):H]=1.
\end{equation}

\begin{proof}[Proof of Theorem~\ref{thm:Serre unif principal}]
Let~$Y_{E(C)}(p)\leq Y_E(p)\leq Y_{A/K}(p)$ be the images of~$Gal(\ol{K}/K\cdot E(C))\leq Gal(\ol{K}/K\cdot E)\leq Gal(\ol{K}/K\cdot E)$ in~$M^{ad}(\A_f)=M_A^{ad}(\A_f)$. Then
\[
[Y(p):Y(p)_E]\leq [K\cdot E:K]\leq [E:C]\leq f(\dim(C)):=2^{\dim(C)}\cdot \dim(C)!.
\]
and~$Y_{E(C)}(p)\leq Y_E(p)$ is a normal subgroup such that~$Y_E(p)/Y_{E(C)}(p)$ is abelian. 

Note that, for every~$p$, the~$p$-adic Lie group~$Y(p)$ has only finitely many open subgroups of index at most~$f:=f(\dim(C))$. The proof of~\cite[App. B (111)]{RY:IHES} applies to every open subgroup of~$Y(p)$.
We have thus
\[
\#{Y_E(p)}^{ab}\leq c_0(A,K,f,p):=
\max\{\# H^{ab}| [Y(p):H]\leq f\}<+\infty.
\]
Therefore
\[
[Y_E(p):Y_{E(C)}]\leq c_0:=c_0(A,K,f,p).
\]
Similarly,
\[
\#{Y_{E(C)}(p)}^{ab}\leq c_1(A,K,f,p):=c_0(A,K,f\cdot c_0,p).
\]

Let~$p\geq p_2(A,K,\dim(C)):=\max\{p_0(A,K);p_1(A,K);f(\dim(C))+1\}$. By~\eqref{eq:113}, we deduce from~$[Y(p):Y_E(p)]\leq f(\dim(C))$ that
\[
Y(p)=Y_E(p).
\]
We can thus deduce from~\eqref{eq:111} that
\[
Y_E(p)=Y_{E(C)}(p),
\]
and thus that
\[
\#Y(p)^{ab}=\#Y_E(p)^{ab}=\#Y_{E(C)}(p)^{ab}=1.
\]
Finally, we have
\begin{multline}
\prod_p \#{Y_{E(C)}(p)}^{ab}\\
\leq
c_2(A,K,\dim(C))
:=
\prod_{p\leq p_2(A,K,\dim(C))}  c_1(A,K,f,p)<+\infty.
\end{multline}
Let~$U_{E(C)}$ be the image of~$\rho_{A\times C}(Gal(\ol{K}/K\cdot E(C))\to M(\A_f)$.
We consider and follow the proof of \cite[App. B, \S B.4.2, p.\,326]{RY:IHES}, for~$M$ the Mumford-Tate group of~$A\times C$ and~$U_{E(C)}$ instead of~$U$. 

We deduce that 
\begin{multline}
[\Gamma:(\Gamma\cap G_1)\times (\Gamma\cap G_2)]=\abs{\Gamma/(\Gamma\cap G_1)}
\\=
\abs{\Gamma/(\Gamma\cap G_2)}
\leq c_2(A,K,\dim(C)).
\end{multline}
Then~\cite[App. B, \S B.4.2, p.\,326]{RY:IHES} gives us
\[
\left[U_{E(C)}:\left(U_{E(C)}\cap M^{der}(\widehat{\Z})\right)\cdot
\left(U_{E(C)}\cap Z(M)(\widehat{\Z})\right)\right
]\leq c_2(A,K,\dim(C)).
\]
We consider the map
\[
\pi:Z(M)\to M^{ab}.
\]
Then~$\ker(\pi)=Z(M^{der})$, where~$M^{der}$ only depends on~$A$. Then for every prime~$p$, we have
\[
\#Z(M^{der})(\Q_p)\leq \#Z(M^{der})(\ol{\Q_p})=c_3(A):=\# Z(M^{der})(\ol{\Q}).
\]
We deduce that
\[
\forall z\in Z(M)(\wh{\Z}), \pi(z)=1\Rightarrow z^{c_3(A)}=1.
\]

We recall that~$C$ is assumed to be a principal CM abelian variety. By Appendix~\ref{sec:appendix reciprocity}, Theorem~\ref{thm:Reciprocity principal} (cf. also \cite[Lem.\,2.7 and its proof]{UY}), there exists~$n:=n(A,\dim(C),K)$ such that for every~$\lambda$ in~$M^{ab}(\wh{\Z})$, the element~$\lambda^{n}$
is in the image of~$U\to M^{ab}(\wh{\Z})$.

Let~$U':=\left(U_{E(C)}\cap M^{der}(\widehat{\Z})\right)\cdot \left(U_{E(C)}\cap Z(M)(\widehat{\Z})\right)$.
Then the image of~$U'\to M^{ab}(\wh{\Z})$ contains
\[
\{\lambda^{n\cdot c_2(A,K,\dim(C))!}|\lambda\in M^{ab}(\wh{\Z})\}.
\]
But this is also the image of~$U''':=U\cap \left(U_{E(C)}\cap Z(M)(\widehat{\Z})\right)$. 
By Lemma~\ref{lem:liftkernel}, we deduce that
\[
\forall z\in Z(M)(\widehat{\Z}), \lambda^{n\cdot c_2(A,K,\dim(C))!\cdot c_3(A)}\in U'''.
\]
This proves~\eqref{eq:thm SUP} for~$e'=n\cdot c_2(A,K,\dim(C))!\cdot c_3(A)$.
\end{proof}

%
%
%
%
%

\subsection{}

We will later use the following.
\begin{lemma}\label{lem:Lang ne}
For every~$e\in \Z_{\geq1}$, there exists~$n_e\in \Z_{\geq1}$ such that we have the following inclusion of ideals
\[
\sum_{\lambda\in \Z^{\times}} (\lambda^e-1)\cdot \wh{\Z}\geq n_e\cdot \wh{\Z}.
\]
\end{lemma}
\begin{proof}Let~$\ell$ be a prime such that~$\ell-1$ does not divide~$e$. As~$\F_\ell^\times\approx \Z/(\ell-1)\Z$, there exists~$\ol{\lambda}\in \F^\times_\ell$ such that~$\ol{\lambda}^e\neq1$. Let~$\lambda\in \wh{\Z}^\times$ 
such that~$\ol{\lambda}=\lambda\pmod{e}$. The image of~$(\lambda^e-1)\cdot \wh{\Z}$ in~${\F}_\ell$ contains the non-zero invertible element~$\ol{\lambda}^e-1$. This image is thus~${\F}_\ell$.
By Nakayama lemma,~$({\lambda}^e-1)\cdot\Z_\ell=\Z_\ell$.

Let~$\ell$ be any prime and~$\lambda=1+\ell$. Then~$((1+\ell)^e-1)\cdot \Z_\ell$ is a non zero ideal of~$\Z_\ell$, and thus is open in~$\Z_\ell$.

We deduce that, with~$n_e:=\prod_{\ell-1\mid e}((1+\ell)^e-1)$, we have 
\[
\sum_{\lambda\in \Z^{\times}} ((\lambda^e-1)-1)\cdot \wh{\Z}\geq \left(\prod_{\ell-1\mid e}((1+\ell)^e-1)\Z_\ell \right)\times\left(\prod_{\ell-1\nmid e} \Z_\ell\right)\geq n_e\cdot \wh{\Z}.
\] This proves Lemma~\ref{lem:Lang ne}.
\end{proof}

\subsection{Kummer theory for abelian varieties}\label{sec:Kummer}

Let~$\Gamma\leq A(K)$ be a subgroup, and let~$\Gamma_0:=\{Q\in A(\ol{K})|\Z_{\geq1}\cdot Q\cap \Gamma\neq \emptyset\}$. We have an exact sequence
\[
0\to A_{tors}\to \Gamma_0\to \Gamma\tens \Q\to 0.
\]
Observe that~$Gal(\ol{K}/K)$ acts trivially on~$\Gamma$, and thus acts trivially on~$\Gamma\tens \Q$ and~$\Gamma\tens\Q/\Z$.
Therefore, we have~$Gal(\ol{K}/K)$-equivariant sequence
\[
0\to A_{tors}\to \Gamma_0/\Gamma \to \Gamma\tens \Q/\Z\to 0.
\]
We assume~$\Gamma\approx\Z^k$ for some~$k\in\Z_{\geq1}$. Then~$\Hom(\Q/\Z,\Gamma\tens\Q/\Z)\simeq \Gamma\tens\wh{\Z}\approx\wh{\Z}^k$ and
\[
\wh{T}_{A,\Gamma}:=
\Hom(\Q/\Z, \Gamma_0/\Gamma)
\approx 
\wh{\Z}^{2\dim(A)+k}
\]
is a~$Gal(\ol{K}/K)$-equivariant extension
\begin{equation}\label{eq:extension}
0\to \wh{T}_{A}\to \wh{T}_{A,\Gamma}\to \Gamma\tens\wh{\Z}\to 0.
\end{equation}
The representation~$\rho_{A,\Gamma}:Gal(\ol{K}/K)\to Aut(\wh{T}_{A,\Gamma})$ stabilises~$\wh{T}_{A}$ and acts trivially on the quotient~$\wh{T}_{A,\Gamma}/\wh{T}_{A}$. There exists thus a map
\[
m=m_{A,K,\Gamma}:Gal(\ol{K}/K)\to\Hom(\Gamma,\wh{T}(A))
\] such that we can represent
\[
\rho_{A,\Gamma}(\sigma)
=
\begin{pmatrix}
\rho_A(\sigma)&m(\sigma)\\
0& 1
\end{pmatrix}
\in
\begin{pmatrix}
\End(\wh{T}(A))&\Hom(\Gamma,\wh{T}(A))\\
\Hom(\wh{T}(A),\Gamma\tens\wh{\Z})& \End(\Gamma\tens\wh{\Z})
\end{pmatrix}.
\]
We call~$(\wh{T}_{A,\Gamma},\rho_{A,\Gamma})$ the \emph{affine Tate module}. The map~$m$ is the cocycle of the~$Gal(\ol{K}/K)$-module~$\Hom(\Gamma,\wh{T}(A))$ coming
from the extension~\eqref{eq:extension}. Its restriction on~$Gal(\ol{K}/K(A_{tors}))=\ker(\rho_A)$, which acts trivially on~$\Hom(\Gamma,\wh{T}(A))$, is an homomorphism of groups~$\ker(\rho_A)\to (\Hom(\Gamma,\wh{T}(A)),+)$.

The object of \emph{Kummer theory for abelian varieties} is to study
\[
\rho_{A,\Gamma}(\ker(\rho_A))=
\begin{pmatrix}
1&m(\ker(\rho_A))\\
0& 1
\end{pmatrix}
=
Im(\rho_{A,\Gamma})
\cap
\begin{pmatrix}
1&\Hom(\Gamma,\wh{T}(A))\\
0& 1
\end{pmatrix},
\]
or to study the map
\begin{equation}\label{eq:Kummer m}
m:Gal(\ol{K}/K(A_{tors}))\to \Hom(\Gamma,\wh{T}(A)).
\end{equation}
We refer to \cite{Ribet}, \cite[App. 2]{Hindry} or \cite[{\S\S} 1--2]{Bertrand} for the following.
\begin{theorem}[Ribet]\label{thm:Kummer}
Let~$A'$ be the abelian variety~$\Hom(\Gamma,A)\approx A^k$, and let~$P\in A'(K)$ be the identity~$\Gamma\to A$.
Let~$[P]:=\ol{\Z\cdot P}^{Zar}$ and~$B=[P]^0$. Then
\[
m(Gal(\ol{K}/K(A_{tors})\cdot K(B))\leq \wh{T}(B)\text{ in }\wh{T}(A')\simeq \Hom(\Gamma,\wh{T}(A))
\]
Let
\begin{equation}\label{}
k(A,K,\Gamma):=[\wh{T}(B):m(Gal(\ol{K}/K(A_{tors})\cdot K(B))]<+\infty.
\end{equation}
If that~$K$ is of finite type over~$\Q$,
then
\[
k(A,K,\Gamma)<+\infty.
\]
\end{theorem}

%
%
%

\subsection{Uniform Kummer theory}

We need the following uniform version.
\begin{theorem}[Uniform Kummer theory on a hybrid orbit]\label{thm:uniform Kummer1}
For every~$d\in\Z_{\geq0}$, there exists~$k(A,K,\Gamma;d)\in\Z_{\geq1}$ such that for every CM abelian variety~$C$ of dimension at most~$d$ 
\begin{equation}\label{eq:unifKummer1}
k(A\times C,K(C),\Gamma)\leq k(A,K,\Gamma;d),
\end{equation}
where~$\Q(C)$ is the field of definition of~$C$, and~$K(C)=K\cdot \Q(C)$.
\end{theorem}
In essence, this is a statement about linear independance, over~$K(A_{tors})$, between~$K(C_{tors})$ and~$K(\Gamma_0)$.

Theorem~\ref{thm:uniform Kummer1} will is a consequence of the following.
\begin{theorem}[Uniform Kummer theory]\label{thm:uniform Kummer2}
For every~$d\in\Z_{\geq0}$, there exists~$\kappa(A,K,\Gamma;d)$ such that for every extension~$[E:K]$ of dimension at most~$d$
\begin{equation}\label{eq:unifKummer2}
k(A,E^{ab},\Gamma)\leq  \kappa(A,K,\Gamma;d)<+\infty.
\end{equation}
\end{theorem}
We will prove Theorem~\ref{thm:uniform Kummer2} in~\S\S\ref{sec:proof 4.10}. We first deduce Theorem~\ref{thm:uniform Kummer1}.

\begin{proof}[Proof of Theorem~\ref{thm:uniform Kummer1} assuming Theorem~\ref{thm:uniform Kummer2}]
Let~$E'$ be the reflex field of the CM-type of~$C$ and let~$E=E'\cdot K$. One has~$[E:K]\leq d:=2^{\dim(C)}\cdot \dim(C)!$, and~$K(C,C_{tors})\subseteq E^{ab}$.
Therefore~$\ker(\rho_{A\times C})=K(C,(A\times C)_{tors})=K(C,C_{\tors},A_{tors})
\subseteq E^{ab}(A_{tors})$. We deduce
\[
m_{A,\Gamma}(Gal(\ol{K}/E^{ab}(A_{tors}))
\subseteq 
m_{A\times C,\Gamma}(\ker(\rho_{A\times C}))
\subseteq 
\wh{T}(B)
\] Therefore~\eqref{eq:unifKummer2} implies~\eqref{eq:unifKummer1} with
\[k(A\times C,K(C_{tors}),\Gamma):= \kappa(A,K,\Gamma;2^{\dim(C)}{\dim(C)}!).\qedhere\]
\end{proof}

\subsubsection{Preliminaries} We adapt~\cite[App. 2]{Hindry}.

\begin{proposition}\label{prop:411}
Let~$A$ be an abelian variety over an extension~$K/\Q$ and let~$N\in\Z_{\geq1}$. There is an injective homomorphism
\[
A(K)/N\cdot A(K)\to H^1(Gal(\ol{K}/K); A[N])
\]
such that for every~$P\in A(K)$ and~$Q\in A(\ol{K})$ such that~$N\cdot Q=P$, the image of~$P+NA(K)$ is the cohomology class of the cocycle~$\sigma\mapsto \sigma(Q)-Q:Gal(\ol{K}/K)\to A[N]$. 

Let~$E/K$ be a finite extension. Then the kernel of~$A(K)/N\cdot A(K)\to A(E)/N\cdot A(E)$ is annihilated by~$[E:K]$.
\end{proposition}
\begin{proof}See~~\cite[App. 2, Lem.~F]{Hindry} for the first statement. Let us prove the last statement. We can embed~$A(K)/N\cdot A(K)$ in~$H^1(Gal(\ol{K}/K); A[N])$ and~$A(E)/N\cdot A(E)$ in~$H^1(Gal(\ol{K}/E); A[N])$. The map~$A(K)/A(K)\to A(E)/N\cdot A(E)$ is induced by the restriction map~$H^1(Gal(\ol{K}/K); A[N])\to H^1(Gal(\ol{K}/E); A[N])$. It suffices to show that the kernel of the latter is annihilated by~$[E:K]$. This is~\cite[I\S\, 2.4, Prop. 9]{LNM5}.
\end{proof}
\begin{proposition}\label{prop:412}
 We consider the context of Theorem~\ref{thm:Kummer}. Let~$E/K$ be a finite extension and let~$n:=n_{e(A,K)\cdot [E:K]!}$ in the notations of Lemma~\ref{lem:Lang ne}. Let~$G'$ and~$H$ be the image and kernel of~$G:=Gal(\ol{K}/E)\to Aut(\wh{T}(A))\times G^{ab}$. 
\begin{enumerate}
\item \label{4121}
For every~$N\in\Z_{\geq1}$, we have a short exact sequence
\begin{equation}\label{eq:truccc}
1\to H^1(G'; B[N]) \to H^1(G; B[N])\to H^1(H;B[N])
\end{equation}
and a natural identification~$H^1(H;B[N])\simeq Hom(H,B[N])$.
\item \label{4122}
For every~$N\in\Z_{\geq1}$, the group~$H^1(G'; B[N])$ is annihilated by~$n$.
\item \label{4123}
For every~$Q\in B(E)$ and~$N\in\Z_{\geq1}$, we have
\[
n\cdot Q\not\in N\cdot B(E)\Rightarrow m_Q(H)\not\subseteq N\cdot \wh{T}(B).
\]
where~$m_Q\in Hom(H,B[N])$ is such that the image of~$Q$ by
\[
B(E)\to B(E)/NB(E)\hookrightarrow H^1(G;B[N])\to H^1(H;B[N])\simeq Hom(H,B[N]).
\]
is the map
\[
H\xrightarrow{m_Q} \wh{T}(B)\to \wh{T}(B)\tens \Z/(N)\simeq B[N].
\]
\end{enumerate}
\end{proposition}

\begin{proof}[Proof of Proposition~\ref{prop:412}]Let us prove~\eqref{4121}. Note that~$H$ acts trivially~$B[N]$. This implies $B[N]^H=B[N]$, and allows us to deduce the exact sequence~\eqref{eq:truccc} from~\cite[Ch.\,I, \S2.6 b)]{LNM5}. This trivial action also implies the identification.

Let us prove~\eqref{4122}.
By Lemma~\ref{lem:Lang ne}, it is enough to prove the claim that for every~$\lambda \in\wh{\Z}^\times$, the group~$H^1(H;A[N])$ is annihilated by~$\lambda^{e'}-1$ with~$e':=e(A,K)\cdot [E:K]!$.
\begin{proof}[Proof of the claim]By Theorem~\ref{thm:Serre eAK}, there exists~$\sigma\in Gal(\ol{K}/K)$ such that~$\sigma$ acts by~$\lambda^{e(A,K)}$ on~$\wh{T}(A)$. Therefore~$\sigma':=\sigma^{[E:K]!}\in Gal(\ol{K}/E)$ and acts by~$\lambda^{e'}$.  Therefore~$\rho_A(\sigma')$ is central in~$Aut(\wh{T}(A))\geq \rho_A(Gal(\ol{K}/K))$. But the image of~$\sigma$ in~$Gal(\ol{K}/E)^{ab}$ is obviously central. Therefore the image of~$\sigma'$ in~$Aut(\wh{T}(A))\times G^{ab}\geq G'$ is central. Note that~$\sigma'$ acts by~$\lambda^{e'}$ on~$B[N]$, and thus by~$\lambda^{e'}$ on~$H^1(H;B[N])$. By Sah's~\cite[App. 2, Lem.~D]{Hindry}, the action of~$\sigma'$ on~$H^1(H;B[N])$ is trivial. This proves the claim.
\end{proof}

Let us prove \eqref{4123}. We can embed~$B(E)/N\cdot B(E)$ in~$H^1(G;B[N])$. By Lemma~\ref{lem:liftkernel} for~$\phi:H^1(G;B[N])\to H^1(H;B[N])$, and~$x=Q+N\cdot B(E)\in B(E)/N\cdot B(E)\leq H^1(G;B[N])$, we have
\[
n\cdot Q+N\cdot B(E)\neq N\cdot B(E)\Rightarrow \phi(Q+N\cdot B(E))\neq 0.
\]
We observe that~$\phi(Q+N\cdot B(E))\in Hom(H,B[N])$ is the image of~$m_Q \in Hom(H,\wh{T}(B))$. We deduce
\[
\phi(Q+N\cdot B(E))\neq 0\Leftrightarrow m_Q(H)\not\subseteq N\cdot \wh{T}(B).
\]
The conclusion follows.

Proposition~\ref{prop:412} is proven.
\end{proof}

\begin{corollary}\label{cor:613}
Let~$\Theta\leq A(K)$ be a subgroup. There exists~$c':=c'(\Theta,A(K))$ such that
\[
\forall N\in\Z_{\geq0},\ker\bigl(\Theta/N\Theta\to A(E)/NA(E)\bigr)
\]
is annihilated by~$c'\cdot [E:K]$.

For all~$N\in\Z_{\geq0}$, the kernel of~$\Theta/N\Theta\to Hom(H,B[N])$ is annihilated by~$c'\cdot [E:K]\cdot n$.
\end{corollary}
\begin{proof} By Proposition~\ref{prop:411}, the kernel~$K_1$ of~$A(K)/NA(K)\to A(E)/NA(E)$ is annihilated by~$[E:K]$. By Proposition~\ref{prop:412} the kernel~$K_2$ of~$A(E)/NA(E)\to  Hom(H,B[N])$ is annihilated by~$n$. The kernel of~$A(K)/NA(K)\to  Hom(H,B[N])$ is an extension~$K_3$ of~$K_2$ by~$K_1$. It follows that~$K_3$ is annihilated by~$n\cdot [E:K]$.

Let~$c'$ be the exponent of~$(A(K)/\Theta)_{tors}$. Theorem~\ref{thm:MWN} implies that the group~$A(K)/\Theta$ is finitely generated, and thus~$c'<+\infty$. Therefore~$x\in A(K), \exists N, N\cdot x\in \Theta\Rightarrow c\cdot x\in \Theta$. Thus~$x\in N\cdot A(K)\Rightarrow c\cdot x\in N\cdot \Theta$. Equivalently, the kernel~$K_4$ of~$\Theta/N\Theta\to A(K)/NA(K)$ is annihilated by~$c'$. The kernel of~$\Theta/N\Theta\to Hom(H,B[N])$ is an extension of~$K_3$ by~$K_4$, and is thus annihilated by~$c'\cdot [E:K]\cdot n$.
\end{proof}
\begin{proposition}\label{prop:preliminaire}
The map
\begin{equation}\label{eq:100}
\alpha\mapsto\alpha(P):\End(B)\mapsto \Theta:=\End(B)\cdot P
\end{equation}
is bijective.

There exists~$c=c(P,B,K)$ such that
\begin{equation}\label{eq:0}
\forall N\in\Z_{\geq1}, \alpha\not\in N\cdot \End(B) \Rightarrow m_{\alpha(P)}(H)\not\subseteq c\cdot [E:K]\cdot N\cdot \wh{T}(B).
\end{equation}
\end{proposition}
\begin{proof}Recall that~$B$ is the algebraic group generated by~$P$. Therefore, for~$\alpha\in \End(B)$ we have~$\alpha(P)=0\Rightarrow \alpha=0$. This proves the injectivity of~\eqref{eq:100}. The surjectivity of~\eqref{eq:100} is immediate.

Let~$c=c(P,B,K):=c'(\Theta,B(K))$ and~$n':=c\cdot [E:K]\cdot n$.
By Corollary~\ref{cor:613}, we have
\begin{equation}\label{eq:1}
n'\cdot \alpha\not\in N\cdot \End(B)\Rightarrow m_{\alpha(P)}(H)\not\subseteq N\cdot \wh{T}(B).
\end{equation}
Substituting~$N$ with~$n'\cdot N$, we obtain
\begin{equation}\label{eq:2}
n'\cdot \alpha\not\in n'\cdot N\cdot \End(B)\Rightarrow m_{\alpha(P)}(H)\not\subseteq n'\cdot N\cdot \wh{T}(B).
\end{equation}

Recall that~$\End(B)$ is a free module of finite rank over~$\Z$. Therefore~$n'\cdot \alpha\not\in n'\cdot N\cdot \End(B)\Leftrightarrow \alpha\not\in N\cdot \End(B)$. We can thus deduce~\eqref{eq:0} from~\eqref{eq:2}.
\end{proof}
\subsubsection{Proof of Theorem~\ref{thm:uniform Kummer2}}\label{sec:proof 4.10}
 We combine Poposition~\ref{prop:preliminaire} with Falting's Theorem in the form~\ref{thm:uniform Faltings} and Theorem~\ref{thm:D1} from Appendix~\ref{app:D}.

Let~$\wh{T}$ be the~$\wh{\Z}$-linear Tate module of~$B$, let~$R=\End(B)$, and let~$\wh{S'}$ be the commutant
of~$R$ in~$\End_{\wh{\Z}}(\wh{T})$. Let~$F(A,K,[E:K])$ be as in Theorem~\ref{thm:uniform Faltings}, and let~$\wh{S}:=\wh{\Z}\cdot Id_{\wh{T}}+F(A,K,[E:K])\cdot \wh{S'}$. According to Theorem~\ref{thm:uniform Faltings},
we have
\[
\wh{S}\subseteq \wh{\Z}\left[\rho_{B}(Gal(\ol{K}/E))\right].
\]
Let~$R^0:=R\tens \Q$. Choose an embedding~$\ol{K}\to \C$, and define~$T_\Z:=H^1(B(\C);\Z)$ and~$T_\Q:=T_\Z\tens\Q$.
Let~$S'$ be the commutant of~$R$ in~$\End_\Z(T_\Z)$ and~$S:=\Z\cdot Id_{T_\Z}+F(A,K,[E:K])\cdot S'$. Then~$\wh{S}\simeq S\tens \wh{\Z}$. Note that~$R$ is an order of the semisimple algebra~$R^0$ and~$S$ is an order of the semisimple algebra~$S^0:=S'\tens\Q$. 

Let~$L$ and~$C(S,R,T)$ be as in Theorem~\ref{thm:D1}. Let~$c(P,B,K)$ be as in Proposition~\ref{prop:preliminaire} and let us define~$n:=c\cdot [E:K]$. Let~$\wh{W}:=m_{P}(H) \leq \wh{T}$. Note that~$m_{\alpha(P)}=\alpha\circ m_{P}(H)$ and therefore
\[
m_{\alpha(P)}(H)\subseteq n \cdot \wh{T}(B)\Leftrightarrow
\alpha(\wh{W})\subseteq n\cdot \wh{T}.
\]

Let~$\ell$ be a prime. We claim that~$W_\ell:=\wh{W}\tens_{\wh{\Z}}\Z_\ell\leq T_\ell:=\wh{T}\tens\Z_\ell$ is open in~$T_\ell$.
\begin{proof} Assume for contradiction that~$U:=W_\ell\tens_{\Z_\ell} \Q_\ell\neq T_\ell\tens_{\Z_\ell} \Q_\ell$. Recall that~$U$ is stable under the action of~$Gal(\ol{E}/E)$. By Faltings' theorem, there exists~$\alpha\in R\tens \Z_\ell$ such that~$\ker(\alpha)=U$. We have~$\alpha\neq 0$ by assumption.
Let~$\alpha_i$ be a sequence in~$R$ converging, in~$R\tens\Z_\ell$, to~$\alpha$. We may assume~$\alpha_i\neq 0$: there exists~$N_i\in\Z_{\geq1}$ such~$\alpha_i\not\in N_i\cdot R$. As~$\alpha_i|_W$ converges to the constant map~$0:W\to T$, ~\eqref{eq:0} implies~$N_i\to 0$ in~$\Q_\ell$ as~$i\to+\infty$. Therefore~$\alpha$ belongs to~$\bigcap N_i\cdot R\tens\Z_\ell=\{0\}$. Therefore~$\alpha=0$ and~$U=\ker(\alpha)=T_\ell \tens_{\Z_\ell} \Q_\ell$, which is a contradiction.
\end{proof}
Let~$W'_\ell\leq T$ be the intersection~$T\cap W_\ell$ in~$T_\ell$. Then~$W'_\ell\leq T$ is of finite index and~$W_\ell\simeq W'_\ell\tens\Z_\ell$.

By Theorem~\ref{thm:D1}, we have
\[
W'_\ell\geq C(S,R,T)\cdot n\cdot T.
\]
Tensoring by~$\Z_\ell$, we get
\[
W_\ell\geq C(S,R,T)\cdot n\cdot T_\ell.
\]

As this is true for every prime~$\ell$, we conclude
\[
\wh{W}:={\prod}_\ell W_\ell\geq {\prod}_\ell C(S,R,T)\cdot n\cdot T_\ell=C(S,R,T)\cdot n\cdot \wh{T}.
\]
This proves Theorem~\ref{thm:uniform Kummer2} with
\begin{equation}\label{eq:Kummer explicite}
\kappa(A,\Gamma;[E:K])=c(P,B,K)\cdot [E:K]\cdot C(S,R,T).
\end{equation}

\subsection{Comparison with a Theorem of Rémond}\label{sec:CD}
The formula~\eqref{eq:Kummer explicite} is an explicit form of Theorem~\ref{thm:uniform Kummer2}. From the proofs of Proposition~\ref{prop:preliminaire} Corollary~\ref{cor:613}, we see that the quantity~$c(P,B,K)$ is the exponent $c'(\End(B)\cdot P,B(K))$ of~$(B(K)/\End(B)\cdot P)_{tors}$. We observe that the quantity
\[
(G(K)\cap (\End(G)\cdot P)_{div})/\End(G)\cdot P
\] 
from~\cite[Theorem~17 (Rémond)]{CD} is equal to~$c'(\End(B)\cdot P,B(K))$, for~$G=A$ and~$B=G_P$.

Therefore, the explicit form~\eqref{eq:Kummer explicite} of Theorem~\ref{thm:uniform Kummer2} implies~\cite[Theorem~17 (Rémond)]{CD} for~$G$ an abelian variety (\cite{CD} also considers $G=A\times GL(1)^N$ for some~$N$).

Our result is stronger:
\begin{itemize}
\item formula~\eqref{eq:Kummer explicite} exhibits an explicit dependency on the field~$E$: the dependency is linear in~$[E:K]$ (cf. Proposition~\ref{prop:411});
\item our results are uniform when~$A$ ranges through a hybrid orbit of abelian varieties (see Theorem~\ref{thm:Gal En}).
\end{itemize}

\section{Proof of Theorem~\ref{thm:Main}}\label{sec:proof}

We consider the setting of Theorem~\ref{thm:Main} and the introduction. 
\subsection{}\label{sec:51}
Let~$A$ be an abelian variety over~$\C$ and let~$P\in A(\C)$. Let~$A_P:=(\ol{\Z\cdot P}^{Zar})^0$ be the abelian subvariety which is the neutral component of the algebraic subgroup generated by~$P$. Replacing~$P$ by the multiple~$n\cdot P$, where~$n=\#\pi_0(\ol{\Z\cdot P}^{Zar})$, we may assume~$\ol{\Z\cdot P}^{Zar}=A_P$.

Let~$S\subseteq \mathbf{A}_\Gamma$ be a special subvariety, and denote by~$\mathcal{A}:=\mathcal{A}_\Gamma|_S\to S$ the restriction to~$S$ of the abelian scheme~$\mathcal{A}_\Gamma\to\mathbf{A}_\Gamma$. 

Let~$\eta$ be a generic point of~$S$, and let~$A_\eta$ be the fiber at~$\eta$, which is an abelian variety over~$\C(S)$. We denote by~$A_{\ol{\eta}}$ the corresponding abelian variety over an algebraic and algebraically closed extension~$\ol{\C(S)}/\C(S)$. Passing to an étale cover~$\mathcal{A}_{\Gamma'}\to\mathcal{A}_{\Gamma}$ for a finite index congruence subgroup,
we have the following property
\begin{equation}
\End(A_\eta/\C(S))=\End(A_{\ol{\eta}}/\ol{\C(S)}).
\end{equation}
Equivalently, property~\eqref{eq:hypmonodromy} holds true for~$\mathcal{A}\to S$.

Let~$s_n=(A_n,\lambda_n)$ a generic sequence in~$S$, viewed as a sequence of polarised abelian varieties with~$\Gamma$-level structure. Let~$Q_n$ be a sequence of points in~$\mathcal{A}$ such that~$Q_n\in A_n$ for all~$n$. We consider~$\mathcal{V}:=\ol{\{Q_n\}}^{Zar(\mathcal{A})}$ and let~$\mathcal{C}$ be a geometric component of~$\mathcal{V}$ of non-zero dimension. Then~$\{n\in\Z_{\geq0}|Q_n\in \mathcal{C}\}$ is infinite.
Substituting~$s_n$ with the corresponding infinite subsequence, we may assume~$\mathcal{V}=\mathcal{C}$.

\subsection{Some finite sets}
For an abelian variety~$A'$, let
\[
\Theta_{A'}:=\{\phi(P)|\phi\in \Hom(A_P,A')\}
\]
and
\[
\Sigma_{A'}:=
\Theta_{A'}+A'_{tors}.
\]
and, for each~$d\in\Z_{\geq1}$,
\[
\Theta_{A',d}:=
\{Q'\in A'|d\cdot Q'\in \Theta_{A'}\}
\]
and
\[
\Xi_{A'}:=\{Q'\in A'|\exists n\in\Z_{\geq1}, n\cdot Q'\in \Theta_{A'}\}=\bigcup_{d\in\Z_{\geq1}}\Theta_{A',d}.
\] 
To each~$Q\in\Xi_{A'}$ we associate
\[
d_Q:=\min\{d\in\Z_{\geq1}|d\cdot Q\in \Sigma_{A'}\},\text{ and }
d'_Q:=\min\{d\in\Z_{\geq1}|d\cdot Q\in \Theta_{A'}\}.
\]
By definition of~$A_P$ and~$\Theta_{A'}$, the map
\[
\varphi\mapsto \varphi(P):\Hom(A_P,A')\to \Theta_{A'}
\]
is a bijection. Therefore~$\Theta_{A'}$ is a free abelian group, and thus is torsion-free.
 Therefore~$\Theta_{A'}\cap A'_{tors}=\{0\}$. That is,~$\Theta_{A'}$ and~$A'_{tors}$
are in direct sum in~$\Sigma_{A'}$.

To~$Q$ we can associate the unique~$\phi_Q,\phi_Q'\in\Hom(A_P,A')$ and~$T_Q\in A_{tors}$ such that
\[
d_Q\cdot Q=\phi_Q(P)+T_Q\text{ and }d'_Q\cdot Q=\phi'_Q(P).
\]
Fix~$\kappa,e\in\Z_{\geq1}$. 
Let~$L_e:=\{\lambda^e|\lambda\in \widehat{\Z}\}\leq \GL(1,\wh{\Z})$ act on~$A'_{tors}$. Let
\begin{align}\label{eq:def E Eprime}
F_Q&:=\phi_Q(A_P[d_Q])\\
\label{def:E Q kappa}
E(Q;\kappa)&:=Q+\kappa\cdot F_Q,\\
\label{def:E Q e}
\text{ and }
E'(d_Q\cdot Q;e)&:=d_Q\cdot Q+L_e\cdot T_Q.
\end{align}

\subsection{Some field extensions} Let~$K/\Q$ be a finitely generated extension such that~$A$ and~$S$ and~$\mathcal{C}$ are definable over~$K$ and~$P\in A(K)$ and~$\End(A/K)\simeq \End(A/\ol{K})$. Let~$n\in\Z_{\geq0}$ be arbitrary and let~$A':=A_{s_n}$. Let~$\phi\in\Hom(A,A')$ be such that~$C:=A'/\phi(A)$ has CM. 

We choose a model of~$A$ over~$K_1:=K$ and consider the action~$\rho_A:\Gal(\ol{K}/K)\to \Aut(A(\ol{K})_{\tors})$. 
Let~$\ol{K}/K_2/K_2'/K_1$ be defined by~$\Gal(\ol{K}/K_2)=\ker(\rho_A)$
and~$\rho(Gal(\ol{K}/K_2'))=\rho(Gal(\ol{K}/K))\cap GL(1,\wh{\Z})$ where~$GL(1,\wh{\Z})$ acts on~$A(\ol{K})_{\tors}$ by homotheties. 

As~$\ker(\phi)$ is stable under the action of~$GL(1,\wh{\Z})$ the quotient~$A\to A/\ker(\phi)$ is defined over~$K_2'$.

We consider the reflex morphism~$r:\Gal(E_*^{ab}/E_*)\to I_E:=\frac{(E\tens \A_f)^\times}{E^\times}$ and the subgroup~$I_\Q:=\frac{\A_f^\times}{\Q^\times}\leq I_E$. We define~$E_*^{ab}/E'_*/E_*$ by~$\Gal(E^{ab}/E'_*)=r^{-1}(I_\Q)$. We note that~$C$ is defined over~$E'$.


Let~$K_3:=E_*^{ab}\cdot K_2$ and~$K_3':=E'_*\cdot K_2'$.

There exists an isogeny
\[
\psi:(A/\ker(\phi))\times C\to A'.
\]
Note that~$(A/\ker(\phi))\times C$ is defined over~$K_3'$ and~$\ker(\psi)$ is stable under~$\Gal(\ol{K}/K_3')$. 
The quotient~$A'':=((A/\ker(\phi))\times C)/\ker(\psi)$ is thus a~$K'_3$-algebraic variety and the quotient map~$(A/\ker(\phi))\times C\to A''$ is defined over~$K'_3$.

We deduce that the image~$[A']=[A'']$ of~$s_n$ by~$S\to \mathbb{A}_g$ belongs to~$\mathbb{A}_g(K_3')$.
Let~$d_\Gamma$ be the degree of the map~$\mathbb{A}_\Gamma\to\mathbb{A}_g$. Then
\[
[K_3(s_n):K_3]\leq [K_3'(s_n):K_3']\leq 
[K(s_n):K([A'])]\leq d_\Gamma<+\infty.
\]

We claim that there exists~$c(\dim(A'))$ such that the isomorphism
\begin{equation}\label{eq:ol psi}
\ol{\psi}:A''\to A'
\end{equation}
induced by~$\psi$ is defined over an extension~$K'_4/K_3'(s_n)$ such that~$[K_4:K_3'(s_n)]\leq c(\dim(A'))$. Let~$K_4=K'_4\cdot K_3$. We have~$[K_4:K_3]\leq [K'_4:K_3']\leq d_\Gamma\cdot c(\dim(A'))$.
\begin{proof}[Proof of the claim]
 Both~$A''$ and~$A'$ are defined over~$L=K_3'(s_n)$. We remark that~$\End(A/K)\simeq \End(A/\ol{K})$ implies~$\End(A''/L)\simeq \End(A''/\ol{L})$, and implies that~$Gal(\ol{L}/L)$ acts trivially on~$Aut(A'')$. Therefore~$H^1(L;Aut(A''))\simeq Hom(Gal(\ol{L}/L);Aut(A''))$. Consider~$Isom(A'',A')$ as a~$Aut(A'')$-torsor. Its cohomology class is a morphism~$\phi:Gal(\ol{L}/L)\to Aut(A'')$. But there is a faithfull representation~$Aut(A'')\hookrightarrow GL(H^1(A(\C);\Z))\simeq GL(2g,\Z)$, and the image of~$\phi$ is finite. By Jordan's theorem,~$[Gal(\ol{L}/L):\ker(\phi)]\leq c(2g)$ for some~$c(2g)\in\Z_{\geq1}$. Therefore~$Isom(A'',A')$ has a~$L'$-rational point~$\psi'$ for some~$L'/L$ with~$[L':L]\leq c(2g)$. Therefore~$\ol{\psi}\in\psi'\cdot Aut(A'')(\ol{L'})=\psi'\cdot Aut(A'')(L)$ is defined over~$L'$.
\end{proof}

\subsubsection{Kummer-Ribet theory}
Let us prove the following.
\begin{theorem}\label{thm:Gal En} Let~$d_\Gamma\in\Z_{\geq1}$ be the degree of the map~$S\to \mathbb{A}_g$ and let~$\kappa'=k(A,K,P\cdot \Z;g)$ be as in Theorem~\ref{thm:uniform Kummer1}.
Let~$A'=A_{s_n}$ and let~$Q'\in\Xi_{A'}$. Let~$\Gal(\ol{K}/K)$ act on~$\mathcal{A}(\ol{K})$.
Then
\[
E(Q',\kappa)\subseteq \Gal(\ol{K}/K) \cdot Q'\cap A'(\ol{K})
\]
with~$\kappa'=(d_\Gamma\cdot c(\dim(A')))!\cdot \kappa$, and~$E(Q';\kappa)$ as in~\eqref{def:E Q kappa}, \eqref{eq:ol psi}.
\end{theorem}
\begin{proof}
By~Theorem~\ref{thm:uniform Kummer1},~$W:=m_P(\Gal(\ol{K}/K_3))\leq \wh{T}_{A_P}$ satisfies
\[
\kappa'\cdot\wh{T}_{A_P}\leq  W.
\]
Therefore, for every~$Q\in A(\ol{K})$ and~$p,q\in\Z_{\geq1}$ such that~$q\cdot Q=p\cdot P$ and~$\gcd(p,q)=1$, we have
\[
Q+\kappa'\cdot A_P[q]=Q+A_P[q/\gcd(q,\kappa')]\subseteq \Gal(\ol{K}/K_3)\cdot Q.
\]
Consider~$Q''\in A''(\ol{K})$,~$T\in A''(\ol{K})_{tors}$ and~$\phi'\in \Hom(A_P,A'')$ such that
\[
Q''=\phi'(Q)+T
\]
and~$\phi'$ is "primitive" in~$\Hom(A,A'')$, i.e.:
$
\forall m\in\Z_{\geq2}, \phi'\not\in m\cdot \Hom(A,A'').
$
Then
\[
\phi'(Q)+\kappa'\cdot \phi'(A_P[q])\subseteq \Gal(\ol{K}/K_3)\cdot \phi'(Q).
\]
As~$T\in A''(K_3)$ is fixed by~$\Gal(\ol{K}/K_3)$ we have also
\[
Q''+\kappa'\cdot \phi'(A_P[q])\subseteq \Gal(\ol{K}/K_3)\cdot Q''.
\]
Note that the action of~$\Gal(\ol{K}/K_3)$ is induced by a morphism~$\mu:\Gal(\ol{K}/K_3)\to \wh{T}_{A''}$
such that~$\kappa'\cdot \phi'(\wh{T}(A_P))$ is contained in the image of~$\mu$. Since~$[K_4:K_3]\leq d_\Gamma\cdot c(\dim(A'))$,
we deduce
\[
(d_\Gamma\cdot c(\dim(A')))!\cdot \kappa'\cdot \phi'(\wh{T}(A_P))\subseteq \mu(\Gal(\ol{K}/K_4)).
\]
This implies
\[
Q'+\kappa\cdot \phi'(A_P[q])\subseteq \Gal(\ol{K}/K_4)\cdot Q'.
\]
We recall that~$\ol{\psi}:A''\simeq A'$ is defined over~$K_4$. Therefore, for the action of~$\Gal(\ol{K}/K_4)$ on~$Q':=\ol{\psi}(Q'')\in A'(\ol{K})$ is induced by the action of~$\Gal(\ol{K}/K)$ on~$\mathcal{A}(\ol{K})$, we have
\[
Q'+\kappa\cdot \phi'(A_P[q])\subseteq \Gal(\ol{K}/K_4)\cdot 
Q'.\qedhere
\]
\end{proof}

\subsubsection{Lang-Serre theory} Let us prove the following.

\begin{theorem}\label{thm:Gal E'n} Let~$d_\Gamma\in\Z_{\geq1}$ be the degree of the map~$S\to \mathbb{A}_g$ and let~$e':=e(A,K,g)$ be as in Theorem~\ref{thm:Serre hybrid}.
Let~$A'=A_{s_n}$ and let~$Q'\in\Sigma_{A'}$. Let~$\Gal(\ol{K}/K)$ act on~$\mathcal{A}(\ol{K})$.
Then
\[
E'(Q';e)\subseteq \Gal(\ol{K}/K) \cdot Q'\cap A(\ol{K})
\]
with~$e=(d_\Gamma\cdot  c(\dim(A'')))!\cdot e'$, and~$E'(Q',e)$ as in~\eqref{def:E Q e}.
\end{theorem}
\begin{proof}Let~$A''$ and~$\ol{\psi}:A''\to A'$ be as in~\eqref{eq:ol psi}. Let~$H$ be the Hilbert class field of the reflex field of~$C$, and let~$E(C)/H$ be a minimal extension over which~$C$ is defined. 
By Theorem~\ref{thm:Serre hybrid}, we have
\[
\forall T\in A''_{tors},
L_{e'}\cdot T\subseteq Gal(\ol{K}/K'_3)\cdot T.
\]
We note that~$[K'_4\cdot E(C):K'_3]\leq d_\Gamma \cdot c(\dim(A''))$ and deduce
\[
\forall T\in A''_{tors},
L_{e}\cdot T\subseteq Gal(\ol{K}/K'_4)\cdot T.
\]
Note that~$\ol{\psi}:A''\to A'$ is defined over~$K'_4$. We deduce
\[
\forall T\in A'_{tors},
L_{e}\cdot T\subseteq Gal(\ol{K}/K'_4)\cdot T.
\]
Let us write~$Q'=R+T$ with~$R\in \Theta_{A'}$ and~$T\in A'_{tors}$. 
We have~$R=\phi(P)$ for some~$\phi\in \Hom(A_P,A')$. We claim that~$\phi$
is defined over~$K'_4$.
\begin{proof}Since~$\ol{\psi}$ is defined over~$K'_4$, it is enough to prove that every~$\phi\in \Hom(A_P,A'')$ is defined over~$K'_4$. From~$\End(A/K)\simeq \End(A/\ol{K})$ we deduce that every abelian subvariety of~$A$ is defined over~$K$. Therefore every abelian subvariety of~$A_P$, of~$A''$ and of~$A_P\oplus A''$ is defined over~$K'_3$. This holds in particular for the graph of~$\phi$. Therefore~$\phi$ is defined over~$K_3'\subseteq K'_4$.
\end{proof}
Therefore
\begin{multline}
Gal(\ol{K}/K'_4\cdot E(C))\cdot (R+T)\\
=R+Gal(\ol{K}/K_4'\cdot E(C))\cdot T\\
\supseteq
R+L_e\cdot T=E'(Q';e).\qedhere
\end{multline}
\end{proof}

%

\subsection{Using equidistribution}
We use the notations from~\S\ref{sec:51}
We consider sequence $(E_n)$ defined as follows. Let $\kappa$ be as in Theorem \ref{thm:Gal En} and 
$$
E_n = E(Q_n, \kappa).
$$
Recall that $\mathcal{C}$ is defined over $K$ and 
$$
Q_n \in \mathcal{C}(\overline{K}).
$$
Therefore 
$$
Gal(\overline{K}/K) \cdot Q_n \subset \mathcal{C}.
$$
By Theorem~\ref{thm:Gal En},
$$
Q_n \in E_n \subset Gal(\overline{K}/K) \cdot Q_n \subset \mathcal{C}.
$$
By taking the Zariski closure in $\mathcal{A}$, we see that
$$
\mathcal{C}=(\bigcup Q_n)^{Zar} \subset (\bigcup E_n)^{Zar} \subset \mathcal{C}.
$$ 

By Theorem \ref{thm:weak}, there exists an abelian subscheme $\mathcal{B}$ of $\mathcal{A}$
such that
$$
\mathcal{C} = \mathcal{C} + \mathcal{B}
$$
and there exists a $d$ such that for every $n$, the image of $E_n$ by the quotient map
$A_n \to A_n/B_{s_n}$ is of cardinality at most $d$.

Let $B'$ be an abelian subvariety of the generic fibre $A_{\eta}$ which is a supplement of $B_{\eta}$ that is,
$A_{\eta}= B' + B_{\eta}$ and there exists $n_2$ such that $B' \cap B_{\eta} \subset A[n_2]$.

Let $\pi \in End(A_{\eta})$ be such that $B'=Im(\pi)$.
Recall that we assumed that $End(A_{\eta})=End(\mathcal{A})$, we can consider the abelian subscheme 
$\mathcal{B'}=Im(\pi)$. For every $s \in S$, $B'_s + B_s = A_s$ and $B'_s \cap B_s \subset A_s[n_2]$.

 Let $F_n = F_Q:=\phi_Q(A_P[d_Q])$ as in \eqref{eq:def E Eprime} with $Q=Q_n$. Let 
$\pi : A_n \to A_n/B_{s_n}$ be the quotient map.
Recall that $E_n = F_n + Q_n$. Therefore,
$$
\# \pi(F_n) = \# \pi(E_n)\leq d
$$   
and thus $\pi(F_n) \subset A_n/B_{s_n}[d!]$.

We now apply Theorem \ref{thm:E1} with $A_1=A$, $A'_1=A_P$, $A_2=A_n$, $A'_2=B_n$, $A''_2=A_n/B_n$,
$Q=Q_n$, $\phi=\phi_{Q_n}$ and $d=d_{Q_n}$ and $d' = gcd(d!,d_{Q_n})$.
The $n_2$ from that theorem is the same as above.

There exists $Q'_n = Q' \in Q_n + B_{s_n} \subseteq \mathcal{C}$ such that
$$
d' n_2 Q' = \phi'_n (P) + T_n
$$
with $\phi' \in Hom(A, A_n)$ and $T_n \in A_{n, tors}$.

Let $\mathcal{C}'$ be the closure of the set of $Q'_n$. We observe that 
$\mathcal{C}' \subset \mathcal{C} \subset \mathcal{C}' + \mathcal{B}$.

Therefore
\[\mathcal{C}' + \mathcal{B} = \mathcal{C} + \mathcal{B}= \mathcal{C}.\]

Note also that the closure of the set $\{d!\cdot n_2 Q'_n\}$ is $ \mathcal{C}'':=d!\cdot n_2 \mathcal{C}'$.

\subsubsection{}

Let $E_n'= E'(d!\cdot n_2 Q'_n, e)$ with $e$ as in Theorem \ref{thm:Gal E'n}. 
As before, we prove
\[
\mathcal{C}''=(\bigcup Q'_n)^{Zar} \subset (\bigcup E'_n)^{Zar} \subset \mathcal{C}''.
\]
By Theorem \ref{thm:weak}, there exists an abelian subscheme $\mathcal{B'}$ of $\mathcal{A}$
such that
$$
\mathcal{C}'' = \mathcal{C}'' + \mathcal{B}'
$$
and there exists a $d'$ such that for every $n$, the image of $E'_n$ by the quotient map
$\pi_n:A_n \to A_n/B'_{s_n}$ is of cardinality at most $d'$. Let us write~$d!\cdot n_2 Q'_n=\phi_n(P)+T_n$
with~$T_n\in {(A'_n)}_{tors}$ and~$\phi\in \Hom(A_P,A')$. Then the image of~$E'_n$
is
\[
\pi_n(E'_n)=\pi_n(\phi_n(P))+L_e\cdot \pi_n(T_n)
\]
and we have
\[
\#\pi_n(E'_n)=\omega_n^{1+o(1)}
\]
where~$\omega_n$ is the order of~$\pi_n(T_n)\in (A_n/B'_{s_n})_{tors}$. We deduce from~$\#\pi_n(E'_n)\leq d'$
that, for some~$d''\in\Z_{\geq1}$, we have~$\forall n\in\Z_{\geq1},\omega_n\leq d''$.

By Lemma~\ref{lem:C3} there exists~$T''_n\in {(A_n)}_{tors}$ of order~$\omega_n$ such that
$T''_n+B'_{s_n}= T_n+B'_{s_n}$. We have thus
\[
d''!\cdot \mathcal{C}''=\ol{\bigcup \phi_n(d''!\cdot P)+ d''!\cdot T_n+B'_{s_n}}^{Zar}=\ol{\bigcup \phi_n(d''!\cdot P)+ B'_{s_n}}^{Zar}
\]
and
\[
d''!\cdot \mathcal{C}''=d''!\cdot \mathcal{C}''+ \mathcal{B'}=
\ol{\bigcup \phi_n(P)}^{Zar}+ \mathcal{B}'.
\]
We deduce
\[
d''!n_2d'\cdot 
\mathcal{C}'=
\ol{\bigcup \phi_n(P)}^{Zar}+ \mathcal{B}'.
\]
and
\[
d''!\cdot n_2\cdot d!\cdot 
\mathcal{C}=
\ol{\bigcup \phi_n(P)}^{Zar}+ \mathcal{B}'+ \mathcal{B}.
\]

This proves Theorem~\ref{thm:Main}.
\appendix
\section{}
\begin{lemma}\label{lem:liftkernel}
Let~$\varphi:X\to Y$ be a homomorphism of abelian groups. Assume that~$\ker(\phi)$ is annihilated by~$n\in\Z_{\geq1}$. Then
\[
\varphi(x)\mapsto n\cdot x:\varphi(X)\to X 
\]
is a well defined homomorphism, and
\[
\forall x\in X, n\cdot x\neq 0\Rightarrow \phi(x)\neq 0.
\]
\end{lemma}
\begin{proof}[Proof of Lemma~\ref{lem:liftkernel}]If~$\varphi(x)=\varphi(x')$, then~$\varphi(x-x')=0$ and thus~$n(x-x')=nx-nx'=0$. For~$\phi(x),\phi(x')\in \varphi(X)$, we have~$\varphi(x)-\varphi(x')=\varphi(x-x')\mapsto n\cdot (x-x')=n\cdot x- n\cdot x'$. The claim follows immediately.
\end{proof}
%
%
%

\section{Some linear algebra and arithmetic stability}\label{App:B}
The aim of this appendix is to prove Theorem~\ref{thm:D1}. This result is used in~\S\ref{sec:proof 4.10} for the proof of Theorem~\ref{thm:uniform Kummer2}.

\subsection{Split case}
\begin{theorem}\label{thm:split linear}
Let~$L/\Q_\ell$ be a finite extension, let~$k,n_1,\ldots,n_k\in\Z_{\geq1}$ and let~$V_1,\ldots,V_k$ be finite dimensional $L$-vector spaces.
Let~$W\leq V= \bigoplus_{i=1}^k L^{n_i}\tens V_i$ be an~$O_L$-submodule.  We consider the action
\[
E:=\prod_{i=1}^k \End(O_L^{n_i})\tens Id_{V_i}\leq \prod_{i=1}^k\End(L^{n_i})\tens \End(V_i)\leq \End(V).
\]
Let~$\lambda\in L$ be such that
\[
\lambda \cdot E\cdot W\leq W.
\]
Then there are~$O_L$-modules~$\Lambda_i\leq V_i$ such that
\[
\lambda\cdot\bigoplus_{i=1}^k O_L^{n_i}\tens_{O_L} \Lambda_i\leq W
\text{ and }
\lambda\cdot W\leq \bigoplus_{i=1}^k O_L^{n_i}\tens_{O_L}\Lambda_i.
\]
\end{theorem}

\begin{proof}[Proof of theorem~\ref{thm:split linear}]
Let~$e_{i,1},\ldots,e_{i,n_i}$ be the standard basis of~${L}^{n_i}$, let
\[
\eps_{i,(m,m')}:= e_{i,m}\times {}^te_{i,m'}\in \End(L^{n_i}):(\lambda_n)_{1\leq n\leq n_i}\mapsto (\delta_{m,n} \lambda_{m'})_{1\leq n\leq n_i}
\]
be the standard elementary matrices and let~
\[
\pi_{i,m}=\eps_{i,(m,m)}\in \End(L^{n_i}):(\lambda_n)_{1\leq n\leq n_i}\mapsto (\delta_{m,n} \lambda_n)_{1\leq n\leq n_i}
\]
be the~$m$-th standard projector.

We can uniquely decompose any~$v\in V$ as
\[
v=\bigoplus_{i=1}^k \sum^{n_i}_{m=1} e_{i,m}\tens v_{i,m}.
\]
Let~$\Lambda_{i,m}:=\{v\in V_i|e_{i,m}\tens v\in W\}$. If~$v\in \Lambda_{i,m}$,
then
\[
e_{i,m'}\tens \lambda \cdot v = \lambda \eps_{i,(m',m)}(e_{i,m}\tens v)\in \lambda\cdot E(W)\leq W.
\]
Therefore~$\lambda\cdot \Lambda_{i,m}\leq \Lambda_{i,m'}$. We deduce
\[
\Lambda_i:=\sum_{i=1}^{m}\Lambda_{i,m'}\geq \bigcap_{i=1}^{n_i} \Lambda_{i,m}\geq \lambda\cdot \Lambda_i.
\]
By construction
\begin{equation}\label{eq:356}
\bigoplus_{i=1}^k O_L^{n_i}\tens \lambda\cdot \Lambda_i
\leq 
\bigoplus_{i=1}^k \sum^{n_i}_{m=1} e_{i,m}\tens \Lambda_{i,m}
\leq W.
\end{equation}
For~$w\in W$, we have
\(
e_i\tens \lambda w_{i,m}=\lambda \pi_{i,m}(w)\in W.
\)
Therefore~$\lambda\cdot w_{i,m}\in \Lambda_{i,m}$ and
\[
\lambda\cdot w=\sum_{1\leq i\leq k,1\leq m\leq n_i}e_{i,m}\tens \lambda\cdot w_{i,m}\in 
\bigoplus_{i=1}^k \sum^{n_i}_{m=1} e_{i,m}\tens \Lambda_{i,m}\leq \bigoplus_{i=1}^k O_L^{n_i}\tens \Lambda_i.
\]
As~$w\in W$ is arbitrary, we conclude
\begin{equation}\label{eq:123}
\lambda\cdot W\leq \bigoplus_{i=1}^k O_L^{n_i}\tens \Lambda_i.
\end{equation}
The Theorem follows from~\eqref{eq:356} and~\eqref{eq:123}.
\end{proof}
\begin{corollary}Let us choose a~$L$-homogeneous norm on each~$V_i$ and equip~$V$ with the corresponding norm. 
Let~$e_i\in (O^\times_L)^{n_i}$ for~$i=1,\ldots,k$. Then
\[
\abs{\lambda}^2\cdot \norm{W}\leq \max_{1\leq i\leq k}\norm{(e_i\tens V_i)\cap W}
\]
\end{corollary}
\begin{proof}For~$e\in O_L^{n_i}$, we have~$\lambda\cdot (e\tens \Lambda_i)\leq (e\tens V_i)\cap W$. Therefore
\[
\abs{\lambda}\cdot \norm{e}\cdot \norm{O_L^{n_i}\tens\Lambda_i}=
\abs{\lambda}\cdot \norm{e\tens \Lambda_i}\leq \norm{(e\tens V_i)\cap W}.
\]
For~$e\in (O_L^\times)^{n_i}$ we have~$\norm{e}=1$.
We deduce
\[
\abs{\lambda}^2 \cdot \norm*{W}\leq \abs{\lambda}\cdot \norm*{\bigoplus_{i=1}^k O_L^{n_i}\tens \Lambda_i}=
\abs{\lambda}\cdot \max_{1\leq i \leq k}\norm*{O_L^{n_i}\tens \Lambda_i}\leq 
\max_{1\leq i \leq k}\norm*{(e\tens V_i)\cap W}.
\]
\end{proof}
\subsubsection{Applications}
Let us mention the case~$\lambda=1$. 
\begin{corollary}
A~$O_L$-submodule~$W\leq \bigoplus_{i=1}^k L^{n_i}\tens V_i$ is of the form~$W=\bigoplus_{i=1}^k O_L^{n_i}\tens_{O_L} \Lambda_i$ if and only if it is~$\prod_{i=1}^kGL(n_i,O_{L})$-invariant.
\end{corollary}
This contains criterions for a lattice to be a direct sum, or a tensor product. For instance the following.
\begin{corollary}Let~$W\leq V:= V_1\tens V_2$ be a lattice. Then~$W$ is of the form~$W_1\tens W_2$ if and only if the maximal compact subgroup~$GL(W)\leq GL(V)$ contains a maximal subgroup of~$GL(V_1)\tens Id_{V_2}$. The latter is necessarily~$GL(W)\cap GL(V_1)\tens Id_{V_2}= GL(W_1)\tens Id_{V_2}$.

The maximal compact subgroups~$GL(W)\leq GL(V)$ containing~$GL(W_1)\tens Id_{V_2}$ are the~$GL(W_1\tens W_2)$.
They form an orbit under~$Id_{V_1}\tens GL(V_2)$.
\end{corollary}

\subsection{Another case}

\begin{theorem}\label{thm:C}
Let~$L_1,\ldots,L_f$ be finite extensions of~$\Q$ or of~$\Q_\ell$, and~$L:=\bigoplus_{i=1}^f L_i$. Let~$Z\leq O_L:=\bigoplus_{i=1}^f O_{L_i}$ be a subring and~$F\in\Z_{\geq1}$ be such that~$F\cdot O_L\leq Z$. 

Let~$O_V:=\bigoplus_{i=1}^f O_{L_i}^{m_i}\leq V:=\bigoplus_{i=1}^f {L_i}^{m_i}$, and~$R:=\End_{O_L}(O_V)=\bigoplus_{i=1}^f \End_{O_{L_i}}(O_{L_i}^{m_i})$, and~$R^0:=\End_{L}(V)\simeq R\tens_{O_L}L$.

Let~$\Lambda\leq O_V$ be a~$Z$-module and~$n\in\Z_{\geq1}$ be such that
\begin{equation}\label{eq:14b}
\forall N\in \Z, \forall \alpha\in R^0, \alpha(\Lambda)\leq N\cdot  O_V\Rightarrow \alpha\in \frac{1}{n}\cdot N\cdot R. 
\end{equation}

Then
\begin{equation}\label{eq:14a}
\Lambda\geq n\cdot F\cdot O_V.
\end{equation}
\end{theorem}
\begin{proof} Let~$\pi_1=(1,0,\ldots,0),\ldots,\pi_f=(0,\ldots,0,1)$ be the units of the~$L_i\leq L$. Let~$\Lambda_i:=O_{L_i}\cdot \pi_i(\Lambda)$. As~$\Lambda$ is a~$Z$-module and~$F\cdot O_{L_i}\cdot \pi_i\in Z$, we have~$F\cdot \Lambda_i=F\cdot O_{L_i}\cdot \pi_i(\Lambda)\leq \Lambda$. Thus
\[
\Lambda\geq F\cdot \bigoplus_{j=1}^f \Lambda_j.
\]
It is sufficient to prove~$\bigoplus_{j=1}^f \Lambda_j\geq  n\cdot O_V.$
We may thus substitute~$F$ with~$1$ if we replace~$\Lambda$ by~$\bigoplus_{j=1}^f \Lambda_j$.

Let us view~$\Lambda^\vee:=\Hom_{O_L}(\Lambda,O_L)$ as a~$O_L$-submodule of~$V^\vee:=\Hom_{L}(\Lambda,L)$. Then~\eqref{eq:14a} is equivalent to
\begin{equation}\label{eq:16}
\Lambda^{\vee}\leq \frac{1}{n}\cdot O^\vee_V.
\end{equation}
Let~$v:=((1,0,\ldots),\ldots,(1,0,\ldots))\in V$. For~$\phi=(\phi_1,\ldots,\phi_f)\in \Lambda^\vee$, the assumption~\eqref{eq:14b} with~$\alpha:=\phi\tens v\in R_0$ and~$N=1$ gives
\[
(\phi_1\tens(1,0,\ldots),\ldots,\phi_f\tens(1,0,\ldots))
=((\phi_1,0,\ldots),\ldots,(\phi_f,0,\ldots))
  \in \frac{1}{n}\cdot R.
\]
This implies~$\phi\in \frac{1}{n}\cdot O_V^\vee$ and proves~\eqref{eq:16}.
\end{proof}
\subsection{General case}\label{app:D}

Let~$K=\Q$ or~$\Q_\ell$ and~$O_K=\Z$ or~$\Z_\ell$ respectively.
Let~$V$ be a $K$-vector space of finite dimension. 
Let~$R^0,S^0\leq \End_K(V)$ be two semisimple~$K$-algebras such that~$S^0$ is the commutant of~$R^0$. Let~$R\leq R^0$ and~$S\leq S^0$ be $O_K$-orders. 

\begin{theorem}\label{thm:D1}
Let~$W,T\leq V$ be a~$O_K$-lattices.
There exists~$C(S,R,T)$, such that
\begin{itemize}
\item if~$S\cdot W\leq W$
\item and 
\begin{equation}\label{eq:D1hyp}
\exists n\in \Z_{\geq1}, \forall N\in\Z_{\geq1},
\forall \alpha \in R^0, 
\alpha(W)\leq N\cdot T\Rightarrow n\cdot \alpha \in N\cdot R.
\end{equation}
\end{itemize}
then 
\[
W\geq C(S,R,T)\cdot n\cdot T.
\]
\end{theorem}

We know that~$R^0$ is the commutant of~$S^0$, and that, for some finite extension~$L/K$, 
one can decompose~$V_L:=V\tens_K L$ as
\[
V_L\simeq \bigoplus_{i=1}^k L^{n_i}\tens L^{m_i}
\]
in such a way that~$R_L:=R\tens L$ and~$S_L:=S\tens L$ act as
\[
\prod_{i=1}^k \End(L^{n_i})\tens Id_{L^{m_i}}\text{ and }
\prod_{i=1}^k  Id_{L^{n_i}} \tens \End(L^{m_i}).
\]
and that~$R$ and~$S$, acting on~$ \bigoplus_{i=1}^k L^{n_i}$ and on~$\bigoplus_{i=1}^k L^{m_i}$,
preserve the integral structures~$ \bigoplus_{i=1}^k O_L^{n_i}$ and on~$\bigoplus_{i=1}^k O_L^{m_i}$.

Let~$R^\bot\leq R^0$ denote the orthogonal of~$R\leq R^0$ with respect to the~$K$-bilinear trace form~$R_0\times R_0\to K$.
Recall that~$R\leq R^\bot$ and that the discriminant, say~$d_R$, of~$R$ is~$[R^\bot:R]$. In particular,~$R\geq d_R\cdot R^\bot$. In~$R_L$ we have
\[
(R\tens O_L)^\bot=R^\bot\tens O_L,
\]
with respect to the~$L$-bilinear trace form~$R_L\times R_L\to L$. We also have
\begin{multline}
d_R\cdot R^\bot\tens O_L
\leq R\tens O_L\leq \prod^k_{i=1}\End(O_L^{n_i})=\left(\prod^k_{i=1}\End(O_L^{n_i})\right)^\bot
\\
\leq (R\tens O_L)^\bot = R^\bot\tens O_L.
\end{multline}
We deduce
\begin{equation}\label{eq:b}
d_R\cdot \prod^k_{i=1}\End(O_L^{n_i})\leq  R\tens O_L\leq \prod^k_{i=1}\End(O_L^{n_i}).
\end{equation}
Similarly,
\[
d_S\cdot \prod^k_{i=1}\End(O_L^{m_i})\leq  S\tens O_L\leq \prod^k_{i=1}\End(O_L^{m_i}).
\]

Let~$W$ and~$T$ satisfy the assumption in Theorem~\ref{thm:D1}. Let~$W'=W\tens O_L$ and~$S'=S\tens O_L$
and~$E= \prod^k_{i=1}Id_{n_i}\tens \End(O_L^{m_i})$ and~$\lambda=d_S$. We have
\[
\lambda\cdot E\cdot W'\leq S'\cdot W' \leq W'.
\]
We may thus apply Theorem~\ref{thm:split linear}. We deduce
\begin{equation}\label{eq:a}
d_S\cdot\bigoplus_{i=1}^k O_L^{n_i}\tens_{O_L}\Lambda_i\leq
W'\leq\frac{1}{d_S}\cdot  \bigoplus_{i=1}^k O_L^{n_i}\tens_{O_L}\Lambda_i
\end{equation}
There exists~$c(T)\in\Z_{\geq1}$ such that
\begin{equation}\label{eq:c}
c(T)\cdot T\tens O_L\leq \bigoplus_{i=1}^k O_L^{n_i}\tens O_L^{m_i}.
\end{equation}

We observe that, for all~$m\in\Z_{\geq1}$,
\begin{align}
\bigoplus_{i=1}^k \Lambda_i
\geq m\cdot \bigoplus_{i=1}^k O_L^{m_i}
&\Rightarrow
\bigoplus_{i=1}^k O_L^{n_i}\tens_{O_L}\Lambda_i
\geq m\cdot \bigoplus_{i=1}^k O_L^{n_i}\tens O_L^{m_i}
\\
&\Rightarrow
W'\geq m\cdot d_S\cdot  \bigoplus_{i=1}^k O_L^{n_i}\tens O_L^{m_i}
\\
&\Rightarrow
W\tens O_L=W'\geq m\cdot c(T)\cdot d_S\cdot T\tens O_L
\\
&\Leftrightarrow 
W\geq m\cdot c(T)\cdot d_S\cdot  T.
\end{align}

Let~$\phi(W,T,R,n)$, where~$n\in\Z_{\geq1}$, denote the formula~\eqref{eq:D1hyp}.

Let~$\Omega:=\{\beta\in \End(V)|\beta(W)\subseteq T\}\simeq W^\vee\tens T$. This is a~$O_K$-lattice of~$\End(V)\simeq V^\vee\tens V$. Then we have
\[
\phi(W,T,R,n)\Leftrightarrow \Omega\cap R^0\subseteq \frac{1}{n}\cdot R.
\]
Let~$R':=R\tens O_L$, and~$T':=T\tens O_L$ and~$\Omega'=W'\tens_{O_L}T'\simeq \Omega\tens O_L\leq \End(V_L)$.
The following is immediate:
\begin{equation}\label{eq:alpha}
\phi(W,T,R,n)\Rightarrow \phi(W',T',R',n).
\end{equation}
Let~$W'':=
\bigoplus_{i=1}^k O_L^{n_i}\tens_{O_L}\Lambda_i$ and~$T'':=\bigoplus_{i=1}^k O_L^{n_i}\tens O_L^{m_i}$
and~$R'':=\prod^k_{i=1}\End(O_L^{n_i})$.
Using~\eqref{eq:a} and~\eqref{eq:b} and~\eqref{eq:c}, we get
\begin{equation}\label{eq:beta}
\phi(W',T',R',n)\Rightarrow \phi(W'',T'',R'',n\cdot d_S\cdot d_R\cdot c(T)).
\end{equation}
Let~$W''':=\bigoplus_{i=1}^k \Lambda_i$ and~$T''':=\bigoplus_{i=1}^k O_L^{m_i}$. 
We also note that
\begin{equation}\label{eq:gamma}
\phi(W'',T'',R'',n\cdot d_S\cdot d_R\cdot c(T))\Leftrightarrow 
\phi(W''',T''',R'',n\cdot d_S\cdot d_R\cdot c(T)).
\end{equation}
Let~$Z=\bigoplus_{i=1}^k  O_L\leq R''=\bigoplus_{i=1}^k \End(O_L^{m_i})$ be the centre of~$R''$.
We note that the~$\Lambda_i$ are~$O_L$-submodules of~$O_L^{m_i}$. Therefore~$W''':=\Lambda_i$ is a~$Z$-submodule
of~$\bigoplus_{i=1}^k O_L^{m_i}$.

We are ready to finish the proof of Theorem~\ref{thm:D1}. By~\eqref{eq:D1hyp}, the formula~$\phi(W,T,R,n)$ holds true.
By~\eqref{eq:alpha},~\eqref{eq:beta},~\eqref{eq:gamma} and the assumption~\eqref{eq:D1hyp}, we have~$\phi(W''',T''',R'',n\cdot d_S\cdot d_R\cdot c(T))$. This implies~\eqref{eq:14b}. Moreover,~$W'''$ is a~$Z$-module with~$F=1$. We can thus apply Theorem~\ref{thm:C}. We deduce
\[
W'''\geq n\cdot d_S\cdot d_R\cdot c(T)\cdot T'''.
\]
From~\eqref{eq:a} and~\eqref{eq:c}, we conclude
\[
W\geq n\cdot {d_S}^2\cdot d_R\cdot c(T)^2\cdot  T.
\]
This proves the conclusion of Theorem~\ref{thm:D1}, with
\begin{equation}
C(S,R,T):={d_S}^2\cdot d_R\cdot c(T)^2.
\end{equation}

\section{On abelian varieties}

Let~$A_P,B\leq A$ be abelian varieties over~$\C$, and let~$P\in A_P$ be such that~$\Z\cdot P$ is Zariski dense in~$A_P$. Let~$\pi:A\to A':=A/B$ be the quotient map.

We recall that there exists~$B'\leq A$ such that~\(
A=B'+B\) and \(B\cap B'\) is finite.
There exists thus~$n\in\Z_{\geq1}$ be such that
\begin{equation}\label{eq:n}
B\cap B'\leq A[n]. 
\end{equation}

\begin{theorem}\label{thm:E1}
 Consider~$Q\in A$ and~$T\in {A}_{tors}$ and~$\phi\in\Hom(A_P,A)$ such that~$d\cdot Q=\phi(P)+T$ for some~$d\in\Z_{\geq1}$. Let~$d'|d$ be such that
\[
\pi\circ \phi(A_P[d])\leq A'[d'].
\]
Then there exists~$Q'\in Q+B$ and~$T'\in {A}_{tors}$ and~$\phi'\in \Hom(A_P,A)$ such that
\[
n\cdot d'\cdot Q'=\phi'(P)+T'.
\]
\end{theorem}
\subsection{Auxiliary Definitions and Results}
We define
\[
\Theta:=\{\phi(P)|\phi\in \Hom(A_P,A)\}\text{ and }
\Theta':=\{\phi(P)|\phi\in \Hom(A_P,A')\}.
\]
For~$d\in\Z_{\geq1}$, let
\[
\Theta_d:=\{Q\in A|d\cdot Q\in \Theta\}\qquad\text{and}\qquad
\Theta'_d:=\{Q\in A'|d\cdot Q\in \Theta'\}.
\]
We claim that the map
\[
\Hom(A_P,A)\to \Theta\quad\text{ and }
\quad \Hom(A_P,A')\to \Theta'
\]
given by~$\phi\mapsto\phi(P)$ are bijective.
\begin{proof}The map are subjective by definition of~$\Theta$ and~$\Theta''$. Observe that~$\phi\mapsto \phi(P)$ is an additive morphism. The injectivity follows from
\[
\phi(P)=0\Rightarrow P\in \ker(\phi)\Rightarrow \ol{P\cdot \Z}^{Zar}\leq \ker(\phi)\Rightarrow \phi(A_P)=0\Rightarrow \phi=0.
\qedhere\]
\end{proof}
Recall that for an algebraic subgroup~$G$ of an abelian variety~$A$, and a morphism~$\phi:A\to B$ of abelian varieties, we have
\[
G\leq \ker(\phi)\Leftrightarrow \exists \psi\in \Hom(A/G,B), \phi=\psi\circ \pi,
\]
where~$\pi:A\to A/G$ is the quotient map. In particular, for~$d\in\Z_{\geq1}$ and~$G=A[d]$, 
then
\begin{equation}\label{eq:lift kernel isogeny}
\forall \phi\in\Hom(A,A), A[d]\leq \ker(\phi)\Leftrightarrow \phi\in d\cdot \Hom(A,A).
\end{equation}

\begin{lemma}\label{lem:C2}
Let~$n$ be as in~\eqref{eq:n}. Let~$\pi_{\star}$ be either of the maps
\[
\phi\mapsto \pi\circ \phi:\Hom(A_P,A)\to\Hom(A_P,A')\text{ and }\phi(P)\mapsto \pi\circ\phi(P):\Theta\to \Theta'.
\]
Then for any~$x\in \coker(\pi_{\star})$, we have~$n\cdot x=0$.
\end{lemma}
\begin{proof}
Let~$B'\leq A$ be as in~\eqref{eq:n}, and let~$p=\pi|_{B'}: B' \to A'$. We recall that~$\ker(p) = B \cap B'\leq B'[n]$ by~\eqref{eq:n}.  It follows that there exists a homomorphism~$\sigma: A' \to B'$ such that $p \circ \sigma=n Id_{A'}$ is the multiplication-by-$n$ map on~$A'$. 

For~$\phi'\in \Hom(A_P,A')$, we have 
\[
n\cdot \phi'=\pi\circ \sigma\circ \phi',
\]
and thus~$\phi'=\pi_\star(\phi)$ for~$\phi=\sigma\circ \phi'$. For~$\phi'(P)\in \Theta'$, we have~$n\cdot \phi'(P)=\pi( \phi(P))$. The lemma follows.
\end{proof}
%

\begin{lemma}\label{lem:C3}
For every~$d\in\Z_{\geq1}$
\begin{enumerate}
\item the map~$A[d]\to A'[d]$ is surjective \label{enum:quotient et torsion}
\item \label{enum:2}
and
\[
\forall Q'\in\Theta_d', \exists \widetilde{Q'}\in \Theta_{n\cdot d}, \pi(\widetilde{Q'})=Q'. 
\]
\end{enumerate}
\end{lemma}
\begin{proof}It is a fact that the exact sequence~$0\to B\to A\to A'\to 0$ is split in the category 
of real Lie groups, and a fortiori as a sequence of abstract groups. The first assertion follows.

Let us prove the second assertion. Let~$Q'\in \Theta'_{d}$ and choose~$\phi'\in \Hom(A_P,A')$ such that~$d Q'=\phi'(P)$. Let~$\sigma$ be as in the proof of Lemma~\ref{lem:C2} and~$\phi:=\sigma\circ \phi'$, and let~$Q$ be such that~$nQ=\sigma(Q')$. 
We have
\[
\pi(n(Q))=\pi(\sigma(Q'))=nQ'.
\]
Therefore~$\pi(Q)-Q'=T'$ for some~$T'\in A'[n]$. By~\eqref{enum:quotient et torsion}, we have~$T'=\pi(T)$ for some~$T\in A[n]$. We have~$Q+T\in \Theta_{nd}$ because
\[
nd (Q+T)=d\sigma(Q')=\sigma(dQ')= \sigma\circ \phi'(P)=\phi(P).
\]
This proves the lemma with~$\wt{Q'}=Q+T$.
\end{proof}
\subsection{Proof of Theorem~\ref{thm:E1}}

From~$\pi\circ \phi(A_P[d])\leq A'[d']$ follows~$A_P[d/d']\leq \ker(\pi\circ \phi)$. This implies that there exists~$\phi':A_P\to A'$ such that~$\pi \circ \phi=(d/d')\cdot \phi'$. Therefore
\[
(d/d')\cdot \phi'(P)=\pi\circ \phi(P)=d\cdot \pi(Q)-\pi(T)=(d/d')\cdot d'\cdot \pi(Q)-\pi(T).
\]
Thus~$(d/d')\cdot (\phi'(P)-d'\cdot \pi(Q))=\pi(T)\in A'_{tors}$, and~$d'\cdot \pi(Q)\in \phi'(P)+A'_{tors}$. There exists thus~$T'\in A'_{tors}$ such that~$d'\cdot (\pi(Q)+T')=\phi'(P)$. Since~$Q':=\pi(Q)+T'\in \Theta'_{d'}$, Lemma~\ref{lem:C3} implies that there exists~$\wt{Q'}\in \Theta_{nd'}$ such that~$\pi(\wt{Q'})=\pi(Q)+T'$. There exists~$\wt{T'}\in A_{tors}$ such that~$\pi(\wt{T'})=T'$. 
Let~$\wt{Q}:=\wt{Q'}-\wt{T'}$. Then~$\pi(\wt{Q})=\pi(Q)$, that is~$\wt{Q}\in Q+B$. Since~$\wt{Q'}\in \Theta_{nd'}$, there exists~$\wt{\phi}\in \Hom(A_P,A)$ such that~$nd'\wt{Q'}=\wt{\phi}(P)$. Let~$\wt{T}=-nd'\wt{T'}\in A_{tors}$. We conclude~$nd'\wt{Q}=\wt{\phi}(P)+\wt{T}$. This proves Theorem~\ref{thm:E1}.

\section{On the Deligne-Shimura Reciprocity law}\label{sec:appendix reciprocity}
\subsection{Uniformity properties of the reflex norm}

Let~$(G,X)$ be a Shimura datum and let~$(H,X_H)$ be a subdatum. We condider~$T:=Z(H)^0$ the connected centre of~$H$ and~$C:=H^{ab}$ the maximal abelian quotient of~$H$.

Let~$x\in X_H$ and consider the special Shimura datum~$(C,x^{ab})$. Let~$\mu:GL(1)_\C\to C_\C$ be the compositum~$\C^\times \to \mathbf{S}(\C)\xrightarrow{x^{ab}}C(\C)$. Let~$Gal(\ol{\Q}/\Q)$ act on~$Y(C_{\ol{\Q}})$ and let~$E:=E(C,x^{ab})$ be such that~$Gal(\ol{\Q}/E)$ is the stabiliser of~$\mu$.

Then~$T_E:=Res_{E/\Q}GL(1)$ has a natural basis~$e_{\sigma(\mu)}$ indexed by the~$\sigma(\mu)\in Gal(\ol{\Q}/\Q)\cdot \mu$. Recall that~$Y(C)\tens\Q$ is $\Q$-linearly generated by the conjugates~$\sigma(\mu)\in Gal(\ol{\Q}/\Q)\cdot \mu$. The natural~$\Z$-linear map
\[
Y(T_E)\simeq \bigoplus_{\sigma(\mu)\in Gal(\ol{\Q}/\Q)\cdot \mu} \Z \cdot e_{\sigma(\mu)}
\to 
\sum_{\sigma(\mu)\in Gal(\ol{\Q}/\Q)\cdot \mu} \Z\cdot \sigma(\mu)\hookrightarrow Y(C)
\]
corresponds to a morphism of $\Q$-algebraic tori~$T_E\to C$. Let $(e^*_{\sigma(\mu)})_{\sigma(\mu)\in Gal(\ol{\Q}/\Q)\cdot \mu}$ be the basis of~$X(T_E)$ dual to the basis~$(e_{\sigma(\mu)})_{\sigma(\mu)\in Gal(\ol{\Q}/\Q)\cdot \mu}$.
By~\cite[Lem. 2.6 and \S2,1]{UY}, there exists~$N\in\Z_{\geq1}$ depending only on~$(G,X)$ and a subset
\[
F\subseteq \bigoplus_{\sigma(\mu)\in Gal(\ol{\Q}/\Q)\cdot \mu} \{-N;\ldots;N\}\cdot e^*_{\sigma(\mu)}\subseteq X(T_E),
\]
 such that the image of
\[
X(C)\to X(T_E)
\]
is generated by~$F$.
\begin{theorem}\label{thm:F1}
 Let~$(G,X)$ be a Shimura datum. 
There exists~$e(G,X)$ such that for every Shimura subdatum~$(H,X_H)$ of~$(G,X)$ and Hodge generic~$x\in X_H$, the reciprocity norm morphism
\[
N:Res_{E(H^{ab},\{x^{ab}\})/\Q}GL(1)\to H^{ab}
\]
satifies
\[
\forall g\in H^{ab}(\wh{\Z}), g^{e(G,X)}\in  N\left(\wh{O_{E(H^{ab},\{x^{ab}\})}}^\times\right),
\]
where~$H^{ab}(\wh{\Z})\leq H^{ab}(\A_f)$ is the maximal compact subgroup.
\end{theorem}
\begin{proof} This follows from the preceding discussion and Proposition~\ref{prop:F3} below.
\end{proof}
\subsection{}
Let~$\pi:T\to S$ be a surjective morphism of~$\Q$-algebraic groups. Recall that there exists~$S'\leq T$ a subtorus such that~$\alpha:S'\to T\to S$ is an isogeny. There exists thus an~$e\in\Z_{\geq1}$ such that~$\ker(\alpha)=S\cap \ker(\pi)\leq T[e]$.


Let~$T(\wh{\Z})\leq T(\A_f)$ and~$S(\wh{\Z})\leq S(\A_f)$ be the maximal compact subgroups. 
\begin{proposition}\label{prop:F2}
We have
\[
\forall s\in S(\wh{\Z}), s^e\in \pi(T(\wh{\Z})).
\]
\end{proposition}
\begin{proof}Note that~$\alpha:S'\to T\to S$ is an isogeny, and that~$\ker(\alpha)\leq S'[e]$. Recall that there exists an isogeny (the dual isogeny)~$\alpha':S\to S'$ such that~$\alpha\circ \alpha':S\to S$ is the morphism~$s\mapsto s^e$. 

Note that~$\alpha':S(\A_f)\to S'(\A_f)$ is a continuous map of separated topological spaces. Therefore~$\alpha'(S(\wh{\Z}))$ is a compact subgroup of~$S'(\A_f)$. We have thus~$\alpha'(S(\wh{\Z}))\leq S'(\wh{\Z})\leq T(\wh{\Z})$. 

We conclude
\[
\pi(T(\wh{\Z}))\geq \alpha(S'(\wh{\Z}))\geq \alpha(\alpha'(S(\wh{\Z})))=\{s^e|s\in S(\wh{\Z})\}.\qedhere
\]
\end{proof}

\begin{proposition}\label{prop:F3}
Consider~$r,N\in\Z_{\geq1}$.
There exists~$e(r,N)$ such that for every transtive action~$Gal(\ol{\Q}/\Q)$ on a set~$E$ of cardinality~$\#\leq r$ and every~$\Z[Gal(\ol{\Q}/\Q)]$-submodule~$Y\leq \Z^E$ generated by a subset of~$\{-N;\ldots;N\}^E$, the morphism of~$\Q$-algebraic tori
\[
T\to S
\] 
corresponding to~$Y\to \Z^E$ satisfies
\[
\forall t\in S(\wh{\Z}), s^e\in \pi(T(\wh{\Z})).
\]
\end{proposition}
\begin{proof}
Choose a bijection~$E\simeq \{1;\ldots; r'\}$ and let~$G$ be the image of~$Gal(\ol{\Q}/\Q)\to S_E\simeq S_{r'}$. For each~$G$ and each subset~$F\subseteq\{-N;\ldots;N\}^E$, let
\[
\alpha(G,F):GL(1)_{\ol{\Q}}^r\to S(G,F)
\]
be the associated morphism of~$\ol{\Q}$-algebraic tori. We claim that there exists a~$G$-invariant sub-torus~$S'(r,F)\leq GL(1)_{\ol{\Q}}$ such that~$\beta(G,F):S'(G,F)\to GL(1)_{\ol{\Q}}^r\to S(G,F)$ is an isogeny.
\begin{proof}Recall that the action of~$G$ on~$\Q^r$ is semisimple. Thre is exists thus~$V(G,F)$ be a~$G$-invariant $\Q$-linear subspace such that~$\Q^r=V(G,F)+\sum_{f\in F} f\cdot Q$ and~$\{0\}=V(G,F)\cap \left(\sum_{f\in F} f\cdot Q\right)$. Let~$GL(1)_{\ol{\Q}}^r\to T(G,F)$ be the corresponding morphism of tori. Then the torus~$S'(r,F)=\ker(\alpha')^0$ is as claimed.
\end{proof}
 Let~$e(G,F)\in\Z_{\geq1}$ be such that~$\forall g\in \ker(\beta(G,F)), g^{e(G,F)}=0$. Since there are finitely many possibilities for~$G$ and~$F$, 
\begin{multline}
e(r,N):={gcd}
\Bigl(
\Bigl\{e(G,F)|r'\in\{1;\ldots;r\},\\
G\leq S_{r'}\text{ a transitive subgroup},\\
F\leq \{1;\ldots;N\}^{r'}
\Bigr\}
\Bigr)
<+\infty.
\end{multline}
Then Proposition~\ref{prop:F3} follows from Proposition~\ref{prop:F2}.
\end{proof}
\subsection{Application to reciprocity and abelian varieties}
Recall that the morphism
\[
r:Gal(\ol{\Q}/E)\to M^{ab}(\A_f)/M^{ab}(\Q)
\]
is given by the Deligne-Shimura reciprocity law. 
If we denote Artin morphism from class field theory by
\[
\alpha:Gal(\ol{\Q}/E)^{ab}\to T_E(\A_f)/T_E(\Q)
\]
where~$T_E:=Res_{E/\Q}GL(1)$, then~$r$ is given by
\[
Gal(\ol{\Q}/E)\to Gal(\ol{\Q}/E)^{ab}\xrightarrow{\alpha} T_E(\A_f)/T_E(\Q)
\xrightarrow{N}
 M^{ab}(\A_f)/M^{ab}(\Q)
\]
where the last morphism is induced by~$N:T_E\to M^{ab}$.

Let~$A$ be an abelian variety over a field of finite type~$K/\Q$ and let~$C$ be abelian variety with CM. Up to isogeny,~$C$ is a principal CM abelian variety, and has a model over the Hilbert class field extension, say~$E(C)/E_\Phi$ of the reflex field of the CM type~$\Phi$ of~$C$.

Let~$M$ be the Mumford-Tate group of~$A\times C$ and~$x:\mathbb{S}\to M_\R$ be the Hodge cocharacter of~$ A\times C$. The action of~$Gal(\ol{\Q}/K\cdot E(C))$ on~$(A\times C)_{\tors}$ induces a representation
\[
Gal(\ol{\Q}/E(C)\cdot K)\to M(\wh{\Z}).
\]
The map~$Gal(\ol{\Q}/E(C)\cdot K)\to M^{ab}(\wh{\Z})/M^{ab}(\Z)\leq M^{ab}({\A_f})/M(\Q)$ is then given by
\[
Gal(\ol{\Q}/E(C)\cdot K)\to Gal(\ol{\Q}/E(M^{ab},\{x^{ab}\}))\xrightarrow{r} M^{ab}({\A_f})/M^{ab}(\Q).
\]

Recall that Galois group~$Gal(E^{ab}/E(C))$ of the hilbert class field extension~$E(C)/E$ is the image of
\[
\wh{O_{E}}^\times/O_E^{\times}\leq T_E(\A_f)/T_E(\Q)\simeq Gal(E^{ab}/E)^{ab}.
\]
Therefore, the image of~$Gal(\ol{\Q}/E(C))$ in~$M^{ab}({\A_f})/M^{ab}(\Q)$
\begin{itemize}
\item is contained in~$M^{ab}(\wh{\Z})/M^{ab}(\Z)$;
\item and by Theorem~\ref{thm:F1} for~$H=M$, there exists~$e$ such that for
all~$\lambda\in M^{ab}(\wh{\Z})/M^{ab}(\Z)$,
the element~$\lambda^e$ is contained in the image of~$Gal(\ol{\Q}/E(C))$ in $M^{ab}(\wh{\Z})/M^{ab}(\Z)$.
\end{itemize}
Note that~$e=e(GSp(2g),H_g)$ depends only on~$g:=\dim(A\times C)$ is bounded when~$A$ is fixed and the dimension of~$C$ is bounded.
 
Note that there exists~$c:\Z_{\geq1}\to \Z_{\geq1}$ such that the cardinality of~$M^{ab}(\Z)$ is at most~$c(\dim(M^{ab}))$. Moreover~
\[
[K\cdot E(C):E(C)]\text{ divides } 
c_K:=[K\cap \ol{\Q}:\Q]<+\infty.
\]
Therefore the image of
\[
Gal(\ol{\Q}/K\cdot E(C))\to M^{ab}(\wh{\Z})/M^{ab}(\Z)
\] 
contains
\[
\{\lambda^{c_K\cdot e}|\lambda \in M^{ab}(\wh{\Z})/M^{ab}(\Z)\}.
\]

By Lemma~\ref{lem:liftkernel}, the image of
\[
Gal(\ol{\Q}/K\cdot E(C))\to M^{ab}(\wh{\Z})
\]
contains
\[
\{\lambda^{c(\dim(M^{ab})!\cdot c_K\cdot e}|\lambda \in M^{ab}(\wh{\Z})\}.
\]
We have proved the following.
\begin{theorem}\label{thm:Reciprocity principal}
Let~$A$ be an abelian variety over a field of finite type~$K/\Q$ and let~$C$ be principal abelian variety with CM. Let~$\Phi$ be the CM type of~$C$, let~$E_\Phi$ be the reflex field of~$\Phi$ and let~$E(C)/E_\Phi$ be the Hilbert class field extension.
  
There exists~$f:=f(\dim(A\times C),[K\cap \ol{\Q}:\Q])\in\Z_{\geq1}$ such that the image of the morphism
\[
Gal(\ol{K}/K\cdot E(C))\to M(\wh{\Z}) \to M^{ab}(\wh{\Z})
\]
induced by the action on~$(A\times C)_{tors}$ contains
\[
\{\lambda^{f}|\lambda \in M^{ab}(\wh{\Z})\}.
\]
\end{theorem}

\end{document}